\documentclass[pdflatex,sn-mathphys-num]{sn-jnl}

\usepackage{graphicx}%
\usepackage{multirow}%
\usepackage{amsmath,amssymb,amsfonts}%
\usepackage{amsthm}%
\usepackage{mathrsfs}%
\usepackage[title]{appendix}%
\usepackage{xcolor}%
\usepackage{textcomp}%
\usepackage{manyfoot}%
\usepackage{booktabs}%
\usepackage{algorithm}%
\usepackage{algorithmicx}%
\usepackage{algpseudocode}%
\usepackage{listings}%
\usepackage{fancyhdr}
\usepackage{setspace}
\usepackage{enumitem}
\usepackage{amssymb}
\usepackage{amsmath}
\usepackage{pst-node}
\usepackage{tikz-cd} 
\usepackage[all,cmtip]{xy}

\usepackage{mathrsfs}
\usepackage{amsthm}

\theoremstyle{thmstyleone}%
\newtheorem{theorem}{Theorem}
\newtheorem{lemma}{Lemma}
\newtheorem{proposition}[theorem]{Proposition}%
\newtheorem{corollary}[theorem]{Corollary}
\theoremstyle{thmstyletwo}%
\newtheorem{remark}{Remark}%

\theoremstyle{thmstylethree}%

\begin{document}

\title[On the generic Simplicity of the spectrum for Connection
Laplacian and $G$-simplicity on Principal Bundles]{On the generic Simplicity of the spectrum for Connection
Laplacian and $G$-simplicity on Principal Bundles}

\author*[1,2]{\fnm{Geovane C.} \sur{Brito}}\email{gcbritomath@gmail.com}

\author[2,3]{\fnm{Marcus A. M.} \sur{Marrocos}}\email{
marcusmarrocos@ufam.edu.br}
\equalcont{These authors contributed equally to this work.}

\affil*[1]{\orgdiv{Departamento de Matemática Aplicada}, \orgname{Instituto de Matemática e Estatística, Universidade de São Paulo}, \orgaddress{\street{Rua do Matão, 1010}, \city{São Paulo}, \postcode{05508‑090}, \state{SP}, \country{Brazil}}}

\affil[2]{\orgdiv{Departamento de Matemática}, \orgname{Universidade Federal do Amazonas}, \orgaddress{\street{Av. General Rodrigo Octávio, 6200}, \city{Manaus}, \postcode{69080‑900}, \state{AM}, \country{Brazil}}}


\abstract{In this paper, we prove that, for a residual set of $C^{k}$ connections defined on a smooth vector bundle $E \to M$, all eigenvalues of the connection Laplacian operator $\mathscr{L}$, acting on the space of sections of $E$, are simple. As an application, we prove that all eigenvalues of the Laplace-Beltrami operator on a compact $G$-principal bundle $P \to M$ are $G$-simple.}




\maketitle

\section{Introduction}
The studies of spectral behavior of Laplace-type operator defined on Riemannian context has a central position at the interface of differential geometry, spectral analysis and mathematical physics
\cite{kato, Arnold1972,kuwabara1982spectra}. The celebrated work of Uhlenbeck \cite{Uhlenbeck}, it has been understood that generic Riemannian metrics on a closed manifold show a rigid spectral behavior. Particulary, the author proved that the eigenvalues of the Beltrami-Laplace operator acting on functions are generically simple.

While the Laplace-Beltrami operator on functions has been extensively studied in several works, For instance \cite{Albert, bando, Craigsutton}, on the other hand, considerably less is known about the spectral behavior of Laplace-type operators defined on smooth vector bundles. This operators arise naturally in geometry and
physics, as the Hodge-Laplace operator acting on differential forms \cite{Peralta,megangier}, and Connection Laplacian operator on smooth vector bundles \cite{kuwabara1982spectra,Zelditchtensor,jung, Bourguignon}. 

The Connection Laplacian operator depends on the
Riemannian metric on the base manifold and linear connection defined on vector bundle $E$. Thus, If we assume that the linear connections $\nabla$ is not compatible with Riemannian metric $g$ on the base manifold, we can pertub each parameter $(g,\nabla)$ independently. 

First, the pertubation of Riemannian metric $g$ on the base manifold $M$ already been extensively analyzed in several contexts, including Riemannian covers,
principal bundles, and warped products. For instance, Zelditch 
\cite{Zelditchtensor}, Marrocos and Gomes \cite{GOMES201921} and Zelditch and Jung \cite{jung}. In this works, generic simplicity typically is obtained from additional hypothesis about the eigensections of Laplace-type operator or dimension of the fibers of smooth vector bundle. Other important work is due to Canzani \cite{Canzani2014}, who studied generic spectral
properties of the conformally covariant operators on vector bundles. She proved that under additional hypothesis about the eigenspaces that there exists a residual subset of the space of Riemannian metric such that, for metric in this subset, the conformally covariant operator has only simple eigenvalues. 

In contrast with previous works mentioned, the present work focuses on perturbations of bundle connections. Recently, Zelditch and Jung \cite{jung} proved the simplicity generic
properties of bundle connection in the special case of Hermitian line bundles over Riemann surfaces. 

The corresponding problem for perturbations of linear connections on smooth vector bundles $E$ with rank $m > 2$ remains open. In this work, we resolve this problem..

Our guiding principle is that, once the linear connection on the vector bundle is allowed to vary freely within a space of $\mathrm{End}(E)$-valued $1$-form on $M$, spectral degeneracies should generically disappear. 

More precisely, let $E$ be a smooth vector bundle of rank $m$ over a compact Riemannian manifold $M$ without boundary endowed with a linear connection $\nabla$ compatible with inner product on the fibers, and let
$$
\mathscr{L} : \Gamma_{L^{2}}(E) \longrightarrow \Gamma_{L^{2}}(E)
$$
denote the Connetion Laplacian operator acting on section of $E$. We study the behavior of the spectrum
of $\mathscr{L}$ along to a one-parameter analytic family of linear connection.

\section*{Statements of Results}

Fixe the Riemannian metric $g$ and vary the linear connection $\nabla$. We will prove the following result:

\begin{theorem}\label{teoremaconexão} Let $E$ be a smooth vector bundle of rank $m$ over a compact Riemannian manifold $M$ without boundary with inner product $\langle\cdot, \cdot \rangle_{E}$ on the fibers. Denote $\mathrm{Conn}(E)$ the space of smooth connections on $E$ compatible with a structural metric $\langle \cdot, \cdot \rangle_{E}$ defined on the fibers. Then, there is a residual subset $\widetilde{\mathrm{Conn}(E)} \subset \mathrm{Conn}(E)$ such that for every $\nabla \in \widetilde{\mathrm{Conn}(E)}$, the Connection Laplacian $\mathscr{L}$ has simple spectrum.
\end{theorem}

As a consequence of the generic simplicity of linear connections, the previous results for Riemannian metrics hold without rigidity or infinitesimal rigidity assumptions on the eigenspaces.

\begin{corollary}
Let $E$ be a vector bundle over compact $M$ without boundary with rank $m\geq 2$ endowed with inner product on the fibers $\langle \cdot, \cdot \rangle_{E}$. Then, the set of the Riemannian metric such that the eigenvalues of operator $\mathscr{L}_{g_{0}, \nabla_{0}}$ is simples is a residual set in $\mathrm{Met}^{k}(g)$.
\end{corollary}

As an application of generic simplicity of linear connections we answered a question regarding the generic $G$-simple spectrum of the real Laplace-Beltrami operator on a $G$-principal bundle which was formulated as number 42 in Yau's list of open problems \cite{Yau2024OpenProblems}.

\begin{theorem}
Let $P$ be a principal $G$-bundle over a compact Riemannian manifold $(M,g)$, endowed with a $G$-invariant $C^{k}$-Riemannian metric $\tilde{g}$ such that $\pi$ is a Riemannian submersion with totally geodesic fibers. Then, the set of Riemannian metric such that the non-zero eigenvalues of Laplace-Beltrami operator $\Delta_{\tilde{g}}$ are $G$-simple is a residual set.
\end{theorem}
\section*{Outline of the Proof}
We will analyze the behavior of these eigenvalue curves under a variations of the connection given by $\nabla(t)=\nabla + t\mathscr{A}$, for some $\mathscr{A} \in \Omega^{1}(M, \mathrm{End}(E))$. Our goal is to prove that these analytic eigenvalue curves do not intersect locally. The argument proceeds as follows.

\begin{enumerate}

\item First, we show that if all eigenvalue curves $\lambda_{i}(t), \lambda_{j}(t)$ coincide for all indices $i,j$, then the \textit{wedge rigid hypothesis} is satisfied. That is,
$$
\lambda_{i}(t) \equiv \lambda_{j}(t) \ \text{for all } i,j 
\quad \Longleftrightarrow \quad 
\text{The eigensection of $\mathscr{L}$ are wedge rigid}
$$

\item Next, assuming that the \textit{wedge rigid hypothesis} is satisfied, we obtain \textit{rigid} and \textit{infinitesimal rigid} hypothesis on the associated eigenspaces. In other words,
$$
\text{wedge rigid hypothesis} \quad \Longrightarrow \quad \text{\textit{rigid} and \textit{infinitesimal rigid} on eigenspaces.}
$$

\item On the other hand, if the \textit{wedge rigid hypothesis}  holds and therefore implies \textit{rigid} and \textit{infinitesimal rigid}, a technical lemma \eqref{TeseMarrocos} shows that these conditions cannot occur simultaneously. Thus, the conjunction of \textit{rigid} and \textit{infinitesimal rigid} leads to a contradiction.

\item Consequently, \textit{wedge rigid hypothesis} cannot hold. Returning to the first step, we conclude that it is not possible that all eigenvalue curves coincide for all pair $(i, j)$. Therefore, there exists at least one pair $(i,j)$ with $i \neq j$ such that
$$
\dot{\lambda}_{i}(t) \neq\dot{\lambda}_{j}(t),
$$
and, in particular, $\lambda_{i}(t) \neq \lambda_{j}(t)$ for some $t$.

\item Hence, we obtain the spliting of the eigenvalue curves. This separation directly implies the density of the set of connections with simple spectrum. The openness follows from the continuous dependence of the eigenvalues on the parameter $\nabla$.
\end{enumerate}

\textbf{Organization of the paper.}
Section \ref{preliminaries} introduces the geometric setting and recalls basic concepts about Connection Laplacian operator on vector bundles. In
Section \ref{technicallemmas}, we obtain the variation formulas for the Connection Laplacian under perturbations of the linear connections and technical lemmas, in special the Splitting Eigenvalue Lemma. These lemmas are fundamental to obtain the split of eigenvalues of the Connection Laplacian operators defined on vector bundle $E$. In the Section \ref{sectionmainresults}
contains the proof of the generic simplicity theorem for linear connections. In the last section, we obtain an answer to Yau's problem 42 and prove the $G$-simplicity for the Laplace-Beltrami operator.

\section{Preliminaries and Notations}\label{preliminaries}

Let $E$ be a smooth vector bundle of rank $m$ over a compact Riemannian manifold $(M,g)$ without boundary of dimension $n$, endowed with the inner product $\langle \cdot, \cdot \rangle_{E}$ defined on the fibers $V$. Denote the space of smooth sections on $E$ by $\Gamma(E)$. We can endowed the smooth vector bundle $E$ with linear connection $\nabla : \Gamma(E) \to \Gamma(T^{*}M \otimes E)$ on $E$ compatible with inner product on the fibers, where 
$$
\Gamma(T^{*}M \otimes E) \cong \Omega^{1}(M, E)
$$
is the set of the $E$-valued $1$-forms defined on Riemannian manifold $M$. 

We will denote by $\mathrm{Conn}(E)$ the space of linear connections of the vector bundle $E$, which is an affine space for the vector space $\Omega^{1}(M, \mathrm{End}(E))$, more precisely
\begin{equation}\label{connectionmodeled}
\mathrm{Conn}(E):= \widetilde{\nabla} + \Omega^1(M; \mathrm{End}(E)),\,\,\mathcal{T}_{\nabla}\mathrm{Conn}(E) \cong   \Omega^1(M; \mathrm{End}(E)),\,\, \text{for all} \ \nabla \in \mathrm{Conn}(E)\end{equation}
where $\Omega^1(M; \mathrm{End}(E))$ is the space of the $\mathrm{End}(E)$-valued  differential $1$-forms defined on $M$. 

In this context, we define the inner product on $\Omega^{1}(M, E)$ as 
\begin{equation}\label{innerproductformsE}
\langle \theta_{1} \otimes u, \theta_{2} \otimes v \rangle_{x} = g^{*}( \theta_{1}, \theta_{2}) \cdot  \langle u, v \rangle_{E}, \quad x \in M ,    
\end{equation}
and extended it linearly to arbitrary element of $\Omega^{1}(M, E)$. Moreover, the inner product $\langle \cdot, \cdot \rangle_{E}$ on the fibers naturally induces a global $L^{2}$-inner product on the space $\Gamma(E)$ of smooth sections:
\begin{equation}\label{L2innerproductonsections}
    \langle u, v\rangle_{L^{2}(E)} \ := \int_{M} \langle u,v \rangle_{E} \ dv_{g}, \quad \text{for all} \, u,v \in \Gamma(E)
\end{equation}
where $dv_{g}$ is volume form induced by the Riemannian metric $g$ defined on $M$. Consequently, define $\Gamma_{L^{2}}(E)$ to be the Hilbert space given by the completion of $\Gamma(E)$ with respect to the global $L^{2}$-inner product defined above in (\ref{L2innerproductonsections}).

In order to obtain the inner product on the space
$$
T^{*}M \otimes \mathrm{End}(E),
$$
we first prove the following proposition in the more general setting of an arbitrary space $\mathrm{End}(V)$. The desired statement will then follow as a direct consequence.

\begin{proposition}\label{lematensorial}
Let $(V,\langle\cdot,\cdot\rangle)$ be a $n$-dimensional real vector space with inner product. For every vectors $a,b\in V$,  consider the tensorial operator
$$
a\otimes b \in \mathrm{End}(V)
$$
by
$$
(a\otimes b)(w):=\langle w,b\rangle\,a , \qquad w\in V .
$$
Then, for every $H\in \mathrm{End}(V)$,
$$
\langle\!\langle H,\,a\otimes b\rangle\!\rangle_{\mathrm{End}(V)}
=
\langle H(b),\,a\rangle ,
$$
where $\langle\!\langle\cdot,\cdot\rangle\!\rangle_{\mathrm{End}(V)}$ denotes the Hilbert-Schmidt inner product on $\mathrm{End}(V)$, defined by
$$
\langle\!\langle S,T\rangle\!\rangle_{\mathrm{End}(V)}
:=
\operatorname{tr}\!\left(S\circ T^{*}\right).
$$
where $T^{*}$ is the formal adjoint of tensor $T$. 

\begin{proof} Let $u_{1}, u_{2} \in V$, and define the following bounded linear operator on $V$ given by 
\begin{equation*}
u_{1}\otimes u_{2}, \quad (u_{1}\otimes u_{2})(v)=\langle v, u_{2}\rangle u_{1} \quad \forall v \in V. 
\end{equation*}

Note that if $u_{1}, u_{2}$ are nonzero, then the rank of $u_{1}\otimes u_{2}$ is one-dimensional. Thus, $u_{1}\otimes u_{2}$ is referred to as a rank-one operator on $V$. 

For all $u_{1},u_{2} \in V$. Then, 
\begin{equation}\label{innerinV}
\langle (a\otimes b)(w_{1}),w_{2}\rangle
=
\langle \langle w_{1},b\rangle a,w_{2}\rangle
=
\langle w_{1},b\rangle\,\langle a,w_{2}\rangle
=
\langle w_{1},(b\otimes a)(w_{2})\rangle .
\end{equation}
where $\langle \cdot, \cdot, \rangle$ is the inner product on the $n$-dimensional vector space. In other words, the relation \eqref{innerinV} prove that the tensorial product $u_{1}\otimes u_{2}$ is a selfadjoint operator. More precisely,
$$
(a\otimes b)^{*}=b\otimes a .
$$
From definition of the Hilbert Schmidt product between tensors, we have
\begin{equation}\label{otimestrace}
\langle\!\langle H,a\otimes b\rangle\!\rangle_{\mathrm{End}(V)}
=
\mathrm{tr}\!\left(H\circ (a\otimes b)^{*}\right)
=
\mathrm{tr}\!\left(H\circ (b\otimes a)\right).
\end{equation}

Consider $\{e_{1},\dots,e_{n}\}$ an orthonormal basis of vector space $V$. From definition of trace operator, it follows that
\begin{equation}\label{traceHba}
\mathrm{tr}\!\left(H\circ (b\otimes a)\right)
= \sum_{i=1}^{n}
\left\langle (H \circ (b \otimes a))(e_i),\, e_i \right\rangle =
\sum_{i=1}^{n}
\left\langle H\!\left(\langle e_{i},a\rangle b\right),e_{i}\right\rangle
\end{equation}
Since $H$ is linear, we have 
$$
H\!\left(\langle e_{i},a\rangle b\right) =  \langle e_{i},a\rangle H(b)
$$ 
and bilinearity of inner product, it follows that
$$\langle \langle e_i, a\rangle\, H(b),\, e_i \rangle
=
\langle e_i, a\rangle\, \langle H(b),\, e_i \rangle .
$$
Thus, we can rewrite the equation \eqref{traceHba} as
$$
\mathrm{tr}\!\left(H\circ (b\otimes a)\right)
= \sum_{i=1}^{n}
\left\langle (H \circ (b \otimes a))(e_i),\, e_i \right\rangle =
\sum_{i=1}^{n}
\left\langle H\!\left(\langle e_{i},a\rangle b\right),e_{i}\right\rangle
=
\sum_{i=1}^{n}
   \langle e_{i},a\rangle\,\langle H(b),e_{i}\rangle.
$$
By linearity of the inner product in the second argument, the sum can be rewritten as
\begin{equation}\label{Hbei}
\left\langle H(b),\sum_{i=1}^{n}\langle e_{i},a\rangle e_{i}\right\rangle .
\end{equation}
Since $\{e_{i}\}_{i=1}^{n}$ is an orthonormal basis, the following relation holds
$$
a=\sum_{i=1}^{n}\langle e_{i},a\rangle e_{i}.
$$
Consequently, we can rewrite \eqref{Hbei} as
\begin{equation*}
\left\langle H(b),\sum_{i=1}^{n}\langle e_{i},a\rangle e_{i}\right\rangle =
\langle H(b),a\rangle .
\end{equation*}
and using the relation \eqref{traceHba}, we obtain the following equality
\begin{equation}\label{traceotimes}
\mathrm{tr}\!\left(H\circ (b\otimes a)\right)
=
\langle H(b),a\rangle .
\end{equation}
Combining the relations \eqref{traceotimes} and \eqref{otimestrace}, it follows that
$$
\langle\!\langle H,a\otimes b\rangle\!\rangle_{\mathrm{End}(V)}
=
\langle H(b),a\rangle ,
$$
which proves the claim. Particulary, for $a=b=v$, we conclude that
$$
\langle\!\langle H,v\otimes v\rangle\!\rangle_{\mathrm{End}(V)}
=
\langle H(v),v\rangle .
$$
for all $v \in V$.
\end{proof}
\end{proposition}
Now, we would like to use the Hilbert-Schmidt inner product $\left\langle\!\left\langle \cdot , \cdot \right\rangle\!\right\rangle_{\mathrm{End}(V)}$
on $\mathrm{End}(V)$ defined in Proposition above to endow $\mathrm{End}(E)$ with the corresponding fiberwise inner product. Moreover, for all $\mathscr{A}, \mathscr{B} \in \Gamma\!\left(T^{*}M \otimes \mathrm{End}(E)\right)$, we can define the $L^{2}$-inner product in this space as
$$
\left\langle\!\left\langle \mathscr{A}, \mathscr{B} \right\rangle\!\right\rangle_{T^{*}M \otimes \mathrm{End}(E)}
:=
\int_{M}
\sum_{i=1}^{n}
\left\langle
\mathscr{A}_{X_{i}},
\mathscr{B}_{X_{i}}
\right\rangle_{\mathrm{End}(E)}
\, dv_{g},
$$
where $\left\{ X_{k}\right\}_{k=1}^{n}$ is a local orthonormal frame of $TM$.

\subsection{Connection Laplacian}
The linear connection $\nabla$ defined on vector bundle $E$ has a formal $L^{2}$–adjoint given by
$$
\nabla^{*} : \Gamma(T^{*}M \otimes E) \longrightarrow \Gamma(E),
$$
satisfying the relation
\begin{equation}\label{adjointformalconnection}
\int_{M} \langle \nabla u, \omega \rangle_{x} \, dv_{g}
\;=\;
\int_{M} \langle u, \nabla^{*}\omega \rangle_{E} \, dv_{g},
\qquad
u\in \Gamma(E),\;
\omega \in \Gamma(T^{*}M \otimes E).
\end{equation}
Consider the second covariant derivative
\begin{equation}\label{covariantderivative1}
\nabla^{2}u \in  \Gamma^{\infty}(T^{*}M\otimes T^{*}M \otimes E), \quad (\nabla^{2}u)(X,Y) = \nabla_{X}(\nabla_{Y}u)- \nabla_{\tilde{\nabla}_{X}Y}u
\end{equation}
where $\tilde{\nabla}$ is the Levi-Civita connection defined on Riemannian manifold $(M,g)$. Consequently, for a local orthonormal frame $\{X_{i}\}_{i=1}^{n}$ defined on $M$, it follows that
\begin{equation}
    \mathrm{tr}_{g}(\nabla^{2}u) = \sum^{n}_{i=1}(\nabla^{2}u)(X_{i}, Y_{i}) = \nabla_{X_{i}}(\nabla_{Y_{i}}u)- \nabla_{\tilde{\nabla}_{X_{i}}Y_{i}}u.
\end{equation}
For all $u, v \in \Gamma(E)$, we have 
$$
\langle \operatorname{tr}\nabla^{2} u, v \rangle
= \sum_{i=1}^{n}
  \left\langle \nabla_{X_{i}} \nabla_{X_{i}} u
   - \nabla_{\nabla_{X_{i}} X_{i}} u,\;
           v \right\rangle.
$$
From compatibility condition of the connection with inner product on the fibers, it follows that
$$
\langle \nabla_{X_{i}} \nabla_{X_{i}} u, v \rangle
 = X_{i} \langle \nabla_{X_{i}} u, v \rangle
   - \langle \nabla_{X_{i}} u, \nabla_{X_{i}} v \rangle.
$$
Now, taking the sum with respect to index $i$, 
$$
\langle \operatorname{tr}\nabla^{2} u, v \rangle
 = \sum_{i=1}^{n} X_{i} \langle \nabla_{X_{i}} u, v \rangle
   - \sum_{i=1}^{n} \langle \nabla_{X_{i}} u, \nabla_{X_{i}} v \rangle
   - \sum_{i=1}^{n} \langle \nabla_{\nabla_{X_{i}} X_{i}} u, v \rangle.
$$
On the other hand, consider the vector field $X$ defined by the relation
\begin{equation}
g(X,Y) = \langle \nabla_{Y} u, v \rangle,
\qquad Y \in TM.
\end{equation}
In particular, we obtain 
$$
g(X,X_{i}) = \langle \nabla_{X_{i}} u, v \rangle \quad \text{and} \quad g(X,\nabla_{X_{i}} X_{i}) = \langle \nabla_{\nabla_{X_{i}} X_{i}} u, v \rangle
$$ and since
$$
X_{i} \langle \nabla_{X_{i}} u, v \rangle
= g(\nabla_{X_{i}} X, X_{i}) + g(X, \nabla_{X_{i}} X_{i})
$$
it follows that
\begin{equation}
\langle \operatorname{tr}\nabla^{2} u, v \rangle_{E}
= \sum_{i=1}^{n} g(\nabla_{X_{i}} X, X_{i})
  - \sum_{i=1}^{n} \langle \nabla_{X_{i}} u, \nabla_{X_{i}} v \rangle_{E}.
\end{equation}
Recall that $\operatorname{div} X = \sum_{i=1}^{n} g(\nabla_{X_{i}} X, X_{i})$, we can conclude that
\begin{equation}
\langle \operatorname{tr}\nabla^{2} u, v \rangle_{E}
= - \langle \nabla u, \nabla v \rangle_{x} + \operatorname{div} X.
\end{equation}
Integrating over closed Riemannian manifold $M$ and using the divergence theorem, we have
\begin{equation}\label{tracecovariant2}
\int_{M} \langle \operatorname{tr}\nabla^{2} u, v \rangle_{E} \, dv_{g}
= - \int_{M} \langle \nabla u, \nabla v \rangle_{x} \, dv_{g}.
\end{equation}
From relation (\ref{adjointformalconnection}), it follows that
\begin{equation}\label{adjointformal2}
\int_{M} \langle \nabla^{*} \nabla u, v \rangle \, dv_{g}
 = \int_{M} \langle \nabla u, \nabla v \rangle \, dv_{g},
\end{equation}
and a comparison of the equations (\ref{tracecovariant2}) and (\ref{adjointformalconnection}), we obtain
\begin{equation}
\int_{M} \langle \operatorname{tr}\nabla^{2} u, v \rangle \, dv_{g}
 = - \int_{M} \langle \nabla^{*} \nabla u, v \rangle \, dv_{g}.
\end{equation}
for all $v \in \Gamma(E)$. Consequently, we have $\nabla^{*} \nabla = - \operatorname{tr}\nabla^{2}$. 

Thus, we are now in a position to define the \textit{Connection Laplacian} operator $\mathscr{L}_{g}: \Gamma_{L^{2}}(E) \to \Gamma_{L^{2}}(E)$ acting on space of the smooth section of $E$ is given by
\begin{equation}\label{connectiondefinition}
\mathscr{L}_{g}= \nabla^{*}\nabla = - \sum_{i=1}^{n} \left( \nabla_{X_{i}}\nabla_{X_{i}} - \nabla_{\mathrm{D}_{X_{i}} X_{i}} \right).   
\end{equation}
It is a positive, formally self-adjoint elliptic operator of second order with discrete spectrum consisting of non-negative eigenvalue acumulating only at infinity, in other words
\begin{equation}
(0 \leq) \ \lambda_{1} \leq \lambda_{2} \leq \cdots \leq \lambda_{k} \leq \cdots \uparrow \infty.
\end{equation}
where each eigenvalue being repeated as many times as its multiplicity indicates. By construction, the differential operator $\mathscr{L}_{g}$ depends analytically on linear connection $\nabla$ defined on vector bundle $E$. Let's denote the spectrum of $\mathscr{L}_{g}$ by 
$$
\operatorname{spec}_{L^2}(\mathscr{L}_{g}) := \left(\lambda_{j}(\nabla) \right)_{j \geq 0} \, ; \quad  \lambda_{0} \geq 0.
$$
The Connection Laplacian operator has discrete spectrum. Then, there exists $\varepsilon>0$ such that $\lambda$ is the only eigenvalue on the open interval $(\lambda -\varepsilon, \lambda +\varepsilon )$.

We also denote by $\mathrm{m}(\lambda, \nabla)$ the geometric multiplicity of an eigenvalue $\lambda$ associated to an eigensection of $\mathscr{L}_{g}$. This multiplicity can be broken under an small pertubation of the linear connection. 

Let $\lambda$ be a fixed nonzero eigenvalue of $\mathscr{L}$ with finite multiplicity $\ell \geq 2$ and $\nabla$ be a linear connection defined on smooth vector bundle $E$ compatible with inner product on the fibers. Since the operator $\mathscr{L}$ depends analitically of the parameter $\nabla$, there exists an open neighborhood $\mathcal{U}$ in $\mathrm{Conn}(E)$ containing $\nabla$ such that for all $\tilde{\nabla}$ close to $\nabla$, we obtain a curve of the eigenvalues $\{\lambda_{i}\}_{i=1}^{\ell}$ of $\mathscr{L}_{g}$ with respecto to $\tilde{\nabla}$ in interval $(\lambda-\varepsilon, \lambda+\varepsilon)$, satisfying the relation
\begin{equation}
    \sum_{i=1}^{\ell} \mathrm{m}(\lambda_{i}, \tilde{\nabla}) = \mathrm{m}(\lambda, \nabla).
\end{equation}
where $\mathrm{m}(\lambda_{i}, \tilde{\nabla})$ denotes the geometric multiplicity of the eigenvalues $\{\lambda_{i}\}_{i=1}^{\ell}$ of the Connection Laplacian with respect to $\tilde{\nabla}$. 

\section{Technical Lemmas}\label{technicallemmas}
Let $\mathrm{Conn}(E)$ be the space of linear connections defined on vector bundle $E$ compatible with inner product $\langle \cdot, \cdot \rangle_{E}$ on the fibers. Fixe $\nabla \in \mathrm{Conn}(E)$ and $\lambda$ a nonzero eigenvalue of the Connection Laplacian operator $\mathscr{L}$ with multiplicity $\ell \geq 2$. Since $\mathrm{Conn}(E)$ is an affine space for the vector space $\Omega^{1}(M, \mathrm{SkewEnd}(E))$,  it is natural to consider an one-parameter analytic family of linear connection given
\begin{equation}\label{curveofconnections}
    \nabla(\cdot) : (-\varepsilon, \varepsilon) \to \mathrm{Conn}(E), \quad \nabla(t) = \nabla + t\mathscr{A}, \quad |t|<\varepsilon, \ \ \text{and} \ \mathscr{A} \in \Omega^{1}(M, \mathrm{SkewEnd}(E)).
\end{equation}
with $\nabla(0)=\nabla$. Let $\{X_{k}\}_{k=1}^{n}$ be a local orthonormal frame of vector fields defined on $(M,g)$. Then, for any vector fields in this frame, it follows that 
\begin{equation}\label{connectioninframe}
\nabla_{X_{i}}(t) = \nabla_{X_{i}} + t \mathscr{A}_{X_{i}}   
\end{equation}
where $\mathscr{A}_{X_{i}} \in \Gamma(\mathrm{SkewEnd}(E))$ is the contraction of $1$-form in $X_{i}$. From definition of Connection Laplacian operator given in \eqref{connectiondefinition}, we have 
\begin{equation}\label{associatedconnectionbundle}
 \mathscr{L}(t) = - \sum_{i=1}^{n}\left(\nabla(t)_{X_{i}}\nabla(t)_{X_{i}} - \nabla(t)_{D_{X_{i}}X_{i}} \right)   
\end{equation}

This allows us to expand each term in the associated Connection Laplacian operator $\mathscr{L}(t)$ in terms of one-parameter family $\nabla(t)=\nabla + t\mathscr{A}$, with $\mathscr{A} \in \Omega^{1}(M, \mathrm{SkewEnd}(E))$.

Using the relation $(\ref{connectioninframe})$, it follows that
\begin{align*}
\nabla(t)_{X_{i}}\nabla(t)_{X_{i}}u &= (\nabla_{X_{i}} + t \mathscr{A}_{X_{i}})(\nabla_{X_{i}}u + t \mathscr{A}_{X_{i}}u)  \\\nonumber
&= \nabla_{X_{i}}(\nabla_{X_{i}}u + t \mathscr{A}_{X_{i}}u) + t\mathscr{A}_{X_{i}}(\nabla_{X_{i}}u + t \mathscr{A}_{X_{i}}u) \\\nonumber
&= \nabla_{X_{i}}\nabla_{X_{i}}u + t\left(\nabla_{X_{i}}(\mathscr{A}_{X_{i}}u) + \mathscr{A}_{X_{i}}(\nabla_{X_{i}}u)\right) + t^{2}\mathscr{A}_{X_{i}}(\mathscr{A}_{X_{i}}u)
\end{align*}
and $\nabla(t)_{D_{X_{i}}X_{i}}u = \nabla_{D_{X_{i}}X_{i}}u + t\mathscr{A}_{D_{X_{i}}X_{i}}u$. Thus, we rewrite the equation \eqref{associatedconnectionbundle} as
\begin{eqnarray*}
\mathscr{L}(t) u &=& - \sum_{i=1}^n 
[
\nabla_{X_{i}} \nabla_{X_{i}} u 
+ t \big(\mathscr{A}_{X_{i}}(\nabla_{X_{i}} u) + \nabla_{X_{i}}(\mathscr{A}_{X_{i}}u)\big)
+ t^{2} \mathscr{A}_{X_{i}}(\mathscr{A}_{X_{i}}u)\\
&\,& - \nabla_{D_{X_{i}} X_{i}} u 
- t \mathscr{A}_{D_{X_{i}} X_{i}} u
]    
\end{eqnarray*}
and rearranging the equation in terms of $t$, we have
\begin{eqnarray}\label{4.8}
\mathscr{L}(t) u 
&=& - \sum_{i=1}^n 
(\nabla_{X_{i}} \nabla_{X_{i}} u - \nabla_{D_{X_{i}} X_{i}} u)
- t \sum_{i=1}^n 
( \mathscr{A}_{X_{i}}(\nabla_{X_{i}} u) + \nabla_{X_{i}}(\mathscr{A}_{X_{i}}u) - \mathscr{A}_{D_{X_{i}} X_{i}} u )\nonumber\\
&-& t^{2} \sum_{i=1}^n \mathscr{A}_{X_{i}}(\mathscr{A}_{X_{i}}u).
\end{eqnarray}

Now, given that the relation \eqref{connectiondefinition} holds,
we rewrite the equation \eqref{4.8} as
$$
\mathscr{L}(t) u = \mathscr{L}u 
- t \sum_{i=1}^n 
\left( \mathscr{A}_{X_{i}}(\nabla_{X_{i}} u) + \nabla_{X_{i}}(\mathscr{A}_{X_{i}}u) - \mathscr{A}_{D_{X_{i}} X_{i}} u \right)
- t^{2} \sum_{i=1}^n \mathscr{A}_{X_{i}}(\mathscr{A}_{X_{i}}u),
$$
and we obtain an explicit variation formula for the Connection Laplacian operator along to one-parameter family of the linear connections $\nabla(t)$ given by

\begin{equation*}
\dot{\mathscr{L}}u 
= - \sum_{i=1}^{n} 
\left(
\mathscr{A}_{X_{i}}\!\left(\nabla_{X_{i}}u\right)
+ \nabla_{X_{i}}\,\left(\mathscr{A}_{X_{i}}u\right)
- \mathscr{A}_{D_{X_{i}}X_{i}}u
\right),
\end{equation*}
for every smooth section $u \in \Gamma(E)$, where 
$\{X_{k}\}_{k=1}^{n}$ is a local orthonormal frame of vector fields defined on $(M,g)$.

The next proposition gives a formula for its derivative along any one-parameter family of linear connection.

\begin{proposition}\label{variationintegral}Let
$$
\nabla(t): (-\varepsilon, \varepsilon) \to \mathrm{Conn}(E), \quad \nabla(t) = \nabla + t\mathscr{A}, \quad |t|<\varepsilon, \ \ \text{and} \ \mathscr{A} \in \Omega^{1}(M, \mathrm{SkewEnd}(E)).
$$
be an one-parameter analytic family of linear connection defined on smooth vector bundle $E$ over a Riemannian manifold $(M,g)$. Then, the infinitesimal variation of eigenvalue for Connection Laplacian operator along to $\nabla(t)$ is given by
\begin{equation*}
\frac{d}{dt}\Big|_{t=0} \lambda(t) = 2\int_{M} \langle\!\langle \mathscr{A},\,  \nabla u\wedge u\rangle\!\rangle_{T^{*}M \otimes \mathrm{End}(E)}\,dv_{g}    
\end{equation*}

\begin{proof} 
Using the relation \eqref{4.8} and differentiating with respect to $t=0$, we have
\begin{align*}
&\int_{M} \frac{d}{dt}\Big|_{t=0} 
   \langle \mathscr{L}(t) u(t), u(t) \rangle_{E} \, dv_{g}
=\sum_{i=1}^{n} \int_{M} \frac{d}{dt}\Big|_{t=0} 
    \langle \nabla(t)_{X_{i}}u(t), \nabla(t)_{X_{i}}u(t) \rangle_{E} \ dv_{g}
\end{align*}
Consequently,
\begin{eqnarray*}
\int_{M} \left\langle \frac{d}{dt}\Big|_{t=0} \lambda(t)u, u \right\rangle_{E}\, dv_{g}
&=&  2 \sum_{i=1}^{n}  \int_{M}\langle \dot{\nabla}_{X_{i}}u, \nabla_{X_{i}} u \rangle_{E}\, dv_{g} \\
&=& 2\sum_{i=1}^{n}  \int_{M} \langle \mathscr{A}_{X_{i}} u, \nabla_{X_{i}} u \rangle_{E}\, dv_{g}.\\
\end{eqnarray*}
From Lemma \ref{lematensorial}, we can rewrite the relation above as 
$$
   \frac{d}{dt}\Big|_{t=0} \lambda(t) = 2\sum_{i=1}^{n} \int_{M} \langle\!\langle \mathscr{A}_{X_{i}}, \nabla_{X_{i}} u \otimes u \rangle\!\rangle_{\mathrm{End}(E)} \, dv_{g}, \qquad \text{for all} \ \mathscr{A} \in  \Omega^{1}(M, \mathrm{SkewEnd}(E))
$$
where $\langle\!\langle \cdot, \cdot\rangle\!\rangle_{E}$ is the Hilbert-Schmidt inner product defined on $\mathrm{SkewEnd}(E)$.

Now, since the $(0,2)$-tensors $\nabla_{X_{i}}u\,\otimes\, u$ admits orthogonal decomposition into symmetric $\mathrm{Sym}(\nabla_{X_{i}}u\,\otimes u)$ and skew-symmetric $\mathrm{Alt}(\nabla_{X_{i}}u\,\otimes u)$ component, for any $\{X_{i}\}_{i=1}^{n}$  local orthonormal framme defined on $M$, it follows that
\begin{eqnarray*}
 \frac{d}{dt}\Big|_{t=0} \lambda(t) &=& 2\sum_{i=1}^{n} \int_{M} \langle\!\langle \mathscr{A}_{X_{i}},\, \nabla_{X_{i}}u\otimes u\rangle\!\rangle_{\mathrm{End}(E)}\,dv_{g}, \nonumber\\
&=& 2\sum_{i=1}^{n} \int_{M}\Big\langle\!\!\Big\langle \mathscr{A}_{X_{i}},\,
\mathrm{Sym}(\nabla_{X_{i}}u\otimes u)
+\mathrm{Alt}(\nabla_{X_{i}}u\otimes u)
\Big\rangle\!\!\Big\rangle_{\mathrm{End}(E)}\,dv_{g},
\quad \nonumber\\
&=& 2\sum_{i=1}^{n} \int_{M}\langle\!\langle \mathscr{A}_{X_{i}},\, \mathrm{Sym}(\nabla_{X_{i}}u\otimes u)\rangle\!\rangle_{\mathrm{End}(E)}
+ \langle\!\langle \mathscr{A}_{X_{i}},\, \mathrm{Alt}(\nabla_{X_{i}}u\otimes u)\rangle\!\rangle_{\mathrm{End}(E)}\,dv_{g},
\nonumber\\
&=& 2\sum_{i=1}^{n} \int_{M}\langle\!\langle \mathscr{A}_{X_{i}},\, \mathrm{Alt}(\nabla_{X_{i}}u\otimes u)\rangle\!\rangle_{\mathrm{End}(E)}\,dv_{g},
\qquad \forall\, \mathscr{A}\in  \Omega^{1}(M, \mathrm{SkewEnd}(E)).
\end{eqnarray*}
Moreover, given that $\mathrm{Alt}(u\otimes \nabla_{X_{i}}u) \in \Lambda^{2}(M)$, we have $\mathrm{Alt}(\nabla_{X_{i}}u\otimes u)= u\wedge\nabla_{X_{i}}u$. Therefore, we obtain the relation 
\begin{eqnarray*}
 \frac{d}{dt}\Big|_{t=0} \lambda(t) &=&2 \sum_{i=1}^{n}\int_{M} \langle\!\langle \mathscr{A}_{X_{i}},\,  \nabla_{X_{i}}u\wedge u\rangle\!\rangle_{\mathrm{End}(E)}\,dv_{g}  \\
 &=&2 \int_{M} \langle\!\langle \mathscr{A},\,  \nabla u\wedge u\rangle\!\rangle_{T^{*}M \otimes \mathrm{End}(E)}\,dv_{g}
\end{eqnarray*}
the proposition follows.
\end{proof}    
\end{proposition}

In this direction, we show that the eigenvalues are unstable under the one-parameter family of linear connections defined in \eqref{curveofconnections}. To this end, we analyze the infinitesimal variation of the eigenvalues obtained in the previous proposition. The following result demonstrates that the eigenvalues are sensitive to perturbations of the linear connection. 

\begin{lemma} Under the same assumptions as in the previous proposition, let $u \in \Gamma(E)$ be an eigensection of the Connection Laplacian operator. Then, the linear functional
$$
\mathcal{F} : \Omega^{1}(M,\mathrm{SkewEnd}(E)) \longrightarrow \mathbb{R}, 
\qquad 
\mathcal{F}(\mathscr{A}) :=2\int_{M} \langle\!\langle \mathscr{A} ,\,  \nabla u\wedge u\rangle\!\rangle_{T^{*}M \otimes \mathrm{End}(E)}\,dv_{g}    
$$
is nontrivial whenever $\nabla u \not\equiv 0$, and hence there exists $\mathscr{A}$ with
$\mathcal{F}(\mathscr{A})\neq 0$. More precisely, since the Proposition \ref{variationintegral} holds, then 
$$
\dot{\lambda}(0)=\frac{d}{dt}\Big|_{t=0} \lambda(t) \sim \mathcal{F}(\mathscr{A}) := 2\int_{M} \langle\!\langle \mathscr{A},\,  \nabla u\wedge u\rangle\!\rangle_{T^{*}M \otimes \mathrm{End}(E)}\,dv_{g}    \neq 0. 
$$ \begin{proof}
    Suppose by contradiction that $\mathcal{F}(\mathscr{A})=0$ for all $\mathscr{A} \in  \Omega^{1}(M, \mathrm{SkewEnd}(E))$. Then, 

$$
   \int_{M} \langle\!\langle \mathscr{A},\,  \nabla u\wedge u\rangle\!\rangle_{T^{*}M \otimes \mathrm{End}(E)}\,dv_{g}     = 0\qquad \text{for all} \ \mathscr{A} \in  \Omega^{1}(M, \mathrm{SkewEnd}(E))
$$
where $\langle\!\langle \cdot, \cdot\rangle\!\rangle_{E}$ is the Hilbert-Schmidt inner product defined on $\mathrm{SkewEnd}(E)$.

Since the inner product is positive and well defined and equality holds for all $\mathscr{A} \in  \Omega^{1}(M, \mathrm{SkewEnd}(E))$, it follows that $\nabla u \wedge u$ must satisfy the relation, 
$$
 \nabla u \wedge u = 0, \qquad \text{for all $u \in \Gamma(E)$}.
$$
Using the definition of the inner product defined on $\Lambda^{2}(M)$, for a vector field $X$ defined on $M$, we obtain the relation
\begin{equation}\label{normwedgezero}
\|\nabla_{X}u\wedge u\|_{\mathrm{End}(E)}^{2}
= \|u\|_{E}^{2}\|\nabla_{X}u\|_{E}^{2} - \langle u,\nabla_{X}u\rangle_{E}^{2} = 0.     
\end{equation}
which implies in $ \|u\|_{E}^{2}\|\nabla_{X_{i}}u\|_{E}^{2}= \langle u,\nabla_{X_{i}}u\rangle_{E}^{2}$ and by Cauchy-Schwarz inequality, the relations holds if only if $u$ and $\nabla_{X} u$ are linearly dependent. More precisely, the equality \eqref{normwedgezero} is satisfied iff there exists a constant $\alpha(X)$ such that 
\begin{equation}\label{connectionAu}
\nabla_{X}u = \alpha(X)u
\end{equation}

From relation \eqref{connectionAu}, we will expand the Connection Laplacian operator with respect to the relation $\nabla_{X} u = \alpha(X)\, u$. For this, consider the expression of the Connection Laplacian operator given by
\begin{equation*}
\mathscr{L} = -\sum_{i=1}^{n}\left(\nabla_{X_{i}}\nabla_{X_{i}}u - \nabla_{D_{X_{i}}X_{i}}u\right),
\end{equation*}
where $\{X_{i}\}_{i=1}^{n}$ is a local orthonormal frame and $D$ is the Levi-Civita connection. Since the relation \eqref{connectionAu} holds, we have
\begin{eqnarray*}
\nabla_{X_{i}}\nabla_{X_{i}}u
&=&
\nabla_{X_{i}}\left(\alpha(X_{i})u\right)\\
&=&X_{i}\left(\alpha(X_{i})\right)u + \alpha(X_{i})\nabla_{X_{i}}\,u, \\
&=&
X_{i}\left(\alpha(X_{i})\right)u + |\alpha(X_{i})|^{2}u,
\end{eqnarray*}
and $\nabla_{D_{X_{i}}X_{i}}u = \alpha(D_{X_{i}}X_{i})\,u$. Thus, we may rewrite \eqref{connectioninframe} as
\begin{equation}\label{laplacianoA}
\mathscr{L}u
=
-\sum_{i=1}^{n}
\left(
X_{i}\left(\alpha(X_{i})\right)
-
\alpha(D_{X_{i}}X_{i})
+
|\alpha(X_{i})|^{2}
\right)u.
\end{equation}
By divergence formula, 
\begin{equation*}
\mathrm{div}(\alpha)
=
\sum_{i=1}^{n}
\left[
X_{i}(\alpha(X_{i}))
-
\alpha(D_{X_{i}}X_{i})
\right],
\end{equation*}
and the modulus of $\alpha$ is given by $|\alpha|^{2} = \sum_{i=1}^{n} |\alpha(X_{i})|^{2}$, we have, 
\begin{equation*}
\mathscr{L} u
= -
\left(\mathrm{div}(\alpha) + |\alpha|^{2}\right)u.
\end{equation*}
Since the connection is compatible with inner product $\langle \cdot, \cdot \rangle_{E}$, the following relation holds:
\begin{equation*}
X\left(\|u\|_{E}^{2}\right)=X\langle u,u\rangle_{E} = 2\langle \nabla_{X}u , u\rangle_{E} 
= 2\alpha(X)\|u\|_{E}^{2},
\end{equation*}
For $u\neq 0$, we have
\begin{equation*}
\alpha(X) = \frac{1}{2\|u\|_{E}^{2}}\, X\left(\|u\|_{E}^{2}\right)
= X\left(\log\|u\|^{2}_{E}\right).
\end{equation*}
Thus, $\alpha = \nabla\left(\log\|u\|^{2}_{E}\right)$ and we obtain the following relations
\begin{equation*}
\mathrm{div}(\alpha) = \Delta\left(\log\|u\|^{2}_{E}\right),
\qquad
|\alpha|^{2} = \left|\nabla\log\|u\|_{E}^{2}\right|^{2}.
\end{equation*}
Substituting into \eqref{laplacianoA}, we obtain the Connection Laplacian operator associated to $\nabla_{X}u=\alpha(X)u$ given by
\begin{equation*}
\mathscr{L} u
= -
\left( \Delta\left(\log\|u\|_{E}^{2}\right) + \left|\nabla\log\|u\|_{E}^{2}\right|^{2} \right)u,
\qquad (u\neq 0).
\end{equation*}
Moreover, given that the relation $\mathscr{L} u = \lambda u$ holds and taking the $L^{2}$-inner product with respect to eigensections $u$ and aplying the Divergence Theorem, we have 
\begin{equation*}
\lambda
= - \frac{1}{\mathrm{vol(M)}}\int_{M}(\left|\nabla\log\|u\|_{E}^{2}\right|^{2})\|u\|_{E}^{2}\,dv_{g},
\qquad (u\neq 0).
\end{equation*}
Given that  $\left|\nabla\log\|u\|_{E}^{2}\right|^{2}$ is positive, from equation above, it follows that $\lambda$ is negative. But from hypothesis, we have $\lambda>0$, which implies that $\lambda =0$, but this is absurd. Therefore, $\mathcal{F}(\mathscr{A})\neq 0$ for all skew-symmetric $\mathscr{A} \in \Omega^{1}(M, \mathrm{SkewEnd}(E))$.
\end{proof}   
\end{lemma}
In the next result, we prove a splitting lemma for the eigenvalues of the Connection Laplacian operator. The proof follows techniques developed by Blecker and Wilson \cite{BleeckerWilson1980}.

\begin{lemma}\label{lemmaconnection} 
Let $\lambda$ be a non-zero eigenvalue of Connection Laplacian operator $\mathscr{L}$ of multiplicity $\ell \geq 2$. Then, there exists $\mathscr{A} \in \Omega^{1}(M,\mathrm{SkewEnd}(E))$ and $\varepsilon > 0$ such that among the perturbed eigenvalues $\lambda_{1}(t), \ldots, \lambda_{\ell}(t)$ of $\mathscr{L}(t)$ there exists a pair $(i,j)$ for which $\lambda_{i}(t) \neq \lambda_{j}(t)$ for all $|t|<\varepsilon$.

\begin{proof}
Let $E$ be a smooth vector bundle over a closed Riemannian manifold $M$ of finite rank $m$. Suppose that the linear connection $\nabla$ is compatible with inner prduct on the fibers, that is, $X\langle u, v\rangle = \langle \nabla_{X}u, v\rangle + \langle u, \nabla_{X}v\rangle$ for all $X \in \Gamma(TM)$ and $u,v \in \Gamma_{L^{2}}(E)$.

Consider a one–parameter family of linear connections defined on smooth vector bundle $E$ given by
$$
\nabla(t)=\nabla+t\mathscr{A},\qquad 
\mathscr{A}\in\Omega^{1}(M,\mathrm{SkewEnd}(E)).
$$
Given $\lambda$ a nonzero eigenvalue of $\mathscr{L}_{g}$ with multiplicity $\ell\ge 2$, by Kato's Lemma, there exist $\ell$ analytic curves of eigenvalues $\lambda_{k}(t)$ and eigensections $u_{k}(t)$ satisfying the relation $\mathscr{L}(t)u_{k}(t)=\lambda_{k}(t)u_{k}(t)$, with 
$\lambda_{k}(0)=\lambda$ for each $k=1, 2, \ldots, \ell$.

From variation formula for eigenvalues obtained in Proposition \ref{variationintegral}, we have
\begin{equation}\label{varformula}
\dot{\lambda}_{i}(0)
=2\int_{M}\langle\!\langle \mathscr{A},u\wedge\nabla\, u\rangle\!\rangle_{T^{*}M \otimes \mathrm{End}(E)}\,dv_{g},\qquad
\dot{\lambda}_{j}(0)
=2\int_{M}\langle\!\langle \mathscr{A}, v\wedge\nabla\, v\rangle\!\rangle_{T^{*}M \otimes \mathrm{End}(E)}\,dv_{g}.
\end{equation}

From hypothesis, 
\begin{equation*}
\dot{\lambda}_{i}(0)=\dot{\lambda}_{j}(0)
\qquad\text{for all }\mathscr{A}\in\Omega^{1}(M,\mathrm{SkewEnd}(E)).
\end{equation*}
Consequently, from relation \eqref{varformula}, we have
\begin{equation*}
\int_{M}\langle\!\langle \mathscr{A}, \nabla\, u \wedge u- \nabla\, v \wedge v\,\rangle\!\rangle_{T^{*}M \otimes \mathrm{End}(E)}\, dv_{g}=0
    \qquad\forall\,\mathscr{A}\in\Omega^{1}(M,\mathrm{SkewEnd}(E)).
\end{equation*}
Therefore, 
$$
\langle\!\langle \mathscr{A}, \nabla\, u \wedge u- \nabla\, v \wedge v\,\rangle\!\rangle_{T^{*}M \otimes \mathrm{End}(E)}=0
$$
for all $\mathscr{A} \in \Omega^{1}(M, \mathrm{SkewEnd}(E))$. 

Since the inner product is positive and well defined and holds for all $\mathscr{A} \in \Omega^{1}(M, \mathrm{SkewEnd}(E))$, the equation above implies that 

\begin{equation}\label{wedge-identity}
\nabla_{X}\, u \wedge u = \nabla_{X}\, v \wedge v\,
\end{equation}
for every vector field $X$ on $M$, or equivalently,  $u\wedge\nabla_{X}u = v\wedge\nabla_{X}v$. 

Thus, since the equation \eqref{wedge-identity} holds and $u,v$ are linearly independents, it follows that
\begin{eqnarray*}
u\wedge\nabla_{X}u
    &=& v\wedge\nabla_{X}v \\
 (u\wedge\nabla_{X}u)\wedge u
    &=& (v\wedge\nabla_{X}v)\wedge u \\ 
    -u\wedge u \wedge \nabla_{X} u & = & v \wedge \nabla_{X} v \wedge u \\
0 &=& v \wedge \nabla_{X} v \wedge u
\end{eqnarray*}
which implies that $u,v, \nabla_{X} v$ are linearly dependents. But since $u,v$ are linearly independents, the unique possibility for the last equation to be satisfied is if $\nabla_{X} v$ is a linear combination of smooth sections $u,v$, more specifically, there exists constants $c(X), d(X)$ such that 
$$
\nabla_{X}v= c(X)u+d(X)v
$$
Similarly, there exists constants $a(X), b(X)$ such that 
\begin{equation*}
\nabla_{X}u= a(X)u+b(X)v    
\end{equation*}
Thus, given the following system, 
\begin{equation}\label{decomp}
\left\{
\begin{aligned}
\nabla_{X} u &= a(X)\,u + b(X)\,v,\\
\nabla_{X} v &= c(X)\,u + d(X)\,v.
\end{aligned}
\right.
\end{equation}
where the coeficients $a, b, c, d$ are constant which depends to vector field $X$ defined on Riemannian manifold $(M,g)$. The system \ref{decomp} immediately implies the relation coeficients $b(X)$ and $c(X)$. Indeed, taking the wedge product of both sides of first equation with $u$ and second equation with $v$, we obtain
\begin{eqnarray*}
    \nabla_{X} u\wedge u &=& a(X)\,u\wedge u + b(X)\,v\wedge u  \quad \Rightarrow  \quad \nabla_{X} u\wedge u = b(X)\,v\wedge u\\
\nabla_{X} v\wedge v &=& c(X)\,u\wedge v + d(X)\,v\wedge v \quad \Rightarrow \quad \nabla_{X} v\wedge v = c(X)\,u\wedge v
\end{eqnarray*}
Since \eqref{wedge-identity} holds, it follows that
$$
c(X)\,u\wedge v = b(X)\,v\wedge u \qquad \Longleftrightarrow  \qquad
c(X)\,u\wedge v = -b(X)\,u\wedge v 
$$
and 
\begin{equation}\label{b+c=0}
c(X)+b(X)=0, \qquad \text{for all} \ X \in \Gamma(TM)     
\end{equation}

Now, we would like to use the relation \eqref{decomp} to expand the  connection Laplacian in terms of coeficients $a(X), b(X), c(X), d(X)$ of the system \eqref{decomp}.

By relation \eqref{connectiondefinition}, it follows that
\begin{equation}\label{laplacian}
\mathscr{L} u
=
-\sum_{i=1}^{n}\left(\nabla_{X_{i}}\nabla_{X_{i}}u - \nabla_{D_{X_{i}}X_{i}}u\right),
\end{equation}
where $D$ is the Levi--Civita connection with respect to Riemannian metric $g$ defined on $M$.

From equations given in \eqref{decomp}, we have
$$
\begin{aligned}
\nabla_{X_{i}}\nabla_{X_{i}}u
&= \nabla_{X_{i}}(a(X_{i}) u + b(X_{i}) v) \\
&= X_{i}(a(X_{i}))\,u + a(X_{i})\nabla_{X_{i}}u
   + X_{i}(b(X_{i}))\,v + b(X_{i})\nabla_{X_{i}}v \\
&= \bigl(X_{i}(a(X_{i}))+a(X_{i})^{2}+b(X_{i}) c(X_{i})\bigr)\,u\\
&+ \bigl(X_{i}(b(X_{i}))+a(X_{i}) b(X_{i}) + b(X_{i}) d(X_{i})\bigr)\,v,
\end{aligned}
$$
and $\nabla_{D_{X_{i}}X_{i}}u
=
a(D_{X_{i}}X_{i})\,u + b(D_{X_{i}}X_{i})\,v$. This allows us to rewrite the relation \eqref{laplacian} as
$$
\begin{aligned}
\mathscr{L} u
&=
-\sum_{i}\Bigl[
\bigl(X_{i}(a(X_{i}))-a(D_{X_{i}}X_{i})+a(X_{i})^{2}+b(X_{i}) c(X_{i})\bigr)u
\\&\qquad\qquad\qquad
+
\bigl(X_{i}(b(X_{i}))-b(D_{X_{i}}X_{i})+a(X_{i}) b(X_{i})+b(X_{i}) d(X_{i})\bigr)v
\Bigr].
\end{aligned}
$$
From divergence formula, it follows that
$$
\mathrm{div}(a)
=
\sum_i\bigl(X_{i}(a(X_{i}))-a(D_{X_{i}}X_{i})\bigr),\qquad
\mathrm{div}(b)
=
\sum_i\bigl(X_{i}(b(X_{i}))-b(D_{X_{i}}X_{i})\bigr),
$$
Consequently,
\begin{equation*}
\mathscr{L} u
=
-\Bigl(\mathrm{div}(a)+\sum_i(a(X_{i})^{2}+b(X_{i}) c(X_{i}))\Bigr)u
-\Bigl(\mathrm{div}(b)+\sum_i(a(X_{i}) b(X_{i}) + b(X_{i}) d(X_{i}))\Bigr)v.
\end{equation*}
Using the relation $\nabla^{*}\nabla u = \lambda u$, we have 
$$
\lambda u
=
-\left(\mathrm{div}(a)+\sum_i(a(X_{i})^{2}+b(X_{i}) c(X_{i}))\right)u
-\left(\mathrm{div}(b)+\sum_i(a(X_{i}) b(X_{i}) + b(X_{i}) d(X_{i}))\right)v.
$$
Reorganizing the equation above, 
$$
0
=
\left(\lambda + \mathrm{div}(a)+\sum_i(a(X_{i})^{2}+b(X_{i}) c(X_{i}))\right)u
+
\left(\mathrm{div}(b)+\sum_i(a(X_{i}) b(X_{i}) + b(X_{i}) d(X_{i}))\right)v.
$$
Since the smooth eigensections $u$ and $v$ are linearly independent, the equation above holds if and only if
\begin{equation*}
\begin{cases}
\begin{aligned}
\mathrm{div}(a) + \sum_i(a(X_{i})^{2}+b(X_{i}) c(X_{i})) &= -\lambda,\\[2mm]
\mathrm{div}(b) + \sum_i(a(X_{i}) b(X_{i}) + b(X_{i}) d(X_{i})) &= 0.
\end{aligned}
\end{cases}
\end{equation*}
Integrating with respect to closed manifold $M$ and using the Divergence Theorem, we obtain the identities

\begin{equation}
\begin{cases}
\begin{aligned}\label{firstzero}
\sum_i\int_{M}(a(X_{i})^{2}+b(X_{i}) c(X_{i})) &= -\lambda,\\[2mm]
\sum_i\int_{M}(a(X_{i}) b(X_{i}) + b(X_{i}) d(X_{i})) &= 0.
\end{aligned}
\end{cases}
\end{equation}
Since the relation \eqref{b+c=0} holds, we can rewrite the system \eqref{firstzero} as 
\begin{equation}\label{eq:U}
\begin{cases}
\begin{aligned}
\sum_i\int_{M}(\left|a(X_{i})\right|^{2}-\left|b(X_{i}\right|^{2})) &= -\lambda,\\[2mm]
\sum_i\int_{M}\left\langle b(X_{i}), a(X_{i}) + d(X_{i})\right\rangle &= 0.
\end{aligned}
\end{cases}
\tag{A}
\end{equation}
To better understand the system \eqref{eq:U}, we would like to explicitly obtain each coeficients of the system \eqref{decomp}. For this, taking the inner products of the first equation in \eqref{decomp} with respect to sections $u$ and $v$, it follows that
\begin{equation*}
\left\{
\begin{aligned}
\langle \nabla_{X} u , u \rangle_{E} &= a(X)\,\|u\|_{E}^{2} + b(X)\,\langle u , v \rangle_{E},\\
\langle \nabla_{X} u , v \rangle_{E} &= a(X)\,\langle u , v \rangle_{E} + b(X)\,\|v\|_{E}^{2}.
\end{aligned}
\right.
\end{equation*}
Rewriting in matrix form, we have
$$
\begin{pmatrix}
\langle \nabla_{X} u , u \rangle_{E} \\
\langle \nabla_{X} u , v \rangle_{E}
\end{pmatrix}
=
\begin{pmatrix}
\|u\|_{E}^{2} & \langle u , v \rangle_{E} \\
\langle u , v \rangle_{E} & \|v\|_{E}^{2}
\end{pmatrix}
\begin{pmatrix}
a(X) \\ b(X)
\end{pmatrix}.
$$
Since $\det(M)\neq 0$, we have
$$
\begin{pmatrix}
a(X) \\ b(X)
\end{pmatrix}
=
\frac{1}{\det(M)}
\begin{pmatrix}
\|v\|_{E}^{2} & -\langle u , v \rangle_{E} \\
-\langle u , v \rangle_{E} & \|u\|_{E}^{2}
\end{pmatrix}
\begin{pmatrix}
\langle \nabla_{X} u , u \rangle_{E} \\
\langle \nabla_{X} u , v \rangle_{E}
\end{pmatrix}.
$$
Now, taking the product of the matrices in hand-side, we obtain the relation
\begin{equation}\label{a-b-general}
\begin{aligned}
a(X) &= 
\frac{1}{\det(M)}\left(
\|v\|_{E}^{2} \langle \nabla_{X} u , u \rangle_{E}
- \langle u , v \rangle_{E} \,\langle \nabla_{X} u , v \rangle_{E}
\right),\\
b(X) &= 
\frac{1}{\det(M)}\left(
-\langle u , v \rangle_{E} \,\langle \nabla_{X} u , u \rangle_{E}
+ \|u\|_{E}^{2}\,\langle \nabla_{X} u , v \rangle_{E}
\right).
\end{aligned}
\end{equation}
Moreover, from hypothesis the linear connection $\nabla$ is compatible with metric structure on the fibers, we have
$$
\langle \nabla_{X} u , u \rangle_{E} = \frac{1}{2}\,X\|u\|_{E}^{2},
$$
and we can rewrite the coeficients obtained in \eqref{a-b-general} as
\begin{equation*}
\begin{aligned}
a(X) &= 
\frac{1}{\det(M)}\left(
\frac{1}{2}\|v\|_{E}^{2} X\|u\|_{E}^{2}
- \langle u , v \rangle_{E} \,\langle \nabla_{X} u , v \rangle_{E}
\right),\\[0.2cm]
b(X) &= 
\frac{1}{\det(M)}\left(
-\frac{1}{2}\langle u , v \rangle_{E} X\|u\|_{E}^{2}
+ \|u\|_{E}^{2}\,\langle \nabla_{X} u , v \rangle_{E}
\right).
\end{aligned}
\end{equation*}
Analogously, taking the inner product of second equation in \eqref{decomp} by $u,v$, we have
\begin{equation*}
\left\{
\begin{aligned}
\langle \nabla_{X} v , u \rangle_{E} &= c(X)\,\|u\|_{E}^{2} + d(X)\,\langle u , v \rangle_{E},\\
\langle \nabla_{X} v , v \rangle_{E} &= c(X)\,\langle u , v \rangle_{E} + d(X)\,\|v\|_{E}^{2}.
\end{aligned}
\right.
\end{equation*}
equivalently,
$$
\begin{pmatrix}
\langle \nabla_{X} v , u \rangle_{E}\\
\langle \nabla_{X} v , v \rangle_{E}
\end{pmatrix}
=
\begin{pmatrix}
\|u\|_{E}^{2} & \langle u , v \rangle_{E} \\
\langle u , v \rangle_{E} & \|v\|_{E}^{2}
\end{pmatrix}
\begin{pmatrix}
c(X) \\ d(X)
\end{pmatrix}.
$$
Thus, we can invert the matrix $M$ and obtain the relation
$$
\begin{pmatrix}
c(X) \\ d(X)
\end{pmatrix}=\frac{1}{\det(M)}\begin{pmatrix}
\|v\|_{E}^{2} & -\langle u , v \rangle_{E} \\
-\langle u , v \rangle_{E} & \|u\|_{E}^{2}
\end{pmatrix}\begin{pmatrix}
\langle \nabla_{X} v , u \rangle_{E}\\
\langle \nabla_{X} v , v \rangle_{E}
\end{pmatrix}
$$
and developing the matrix product in hand-side, if follows that
\begin{equation}\label{c-d-general}
\begin{aligned}
c(X) &= \frac{1}{\det(M)}\left(
\|v\|_{E}^{2}\,\langle \nabla_{X} v , u \rangle_{E}
- \langle u , v \rangle_{E} \,\langle \nabla_{X} v , v \rangle_{E}
\right),\\[0.2cm]
d(X) &= \frac{1}{\det(M)}\left(
-\,\langle u , v \rangle_{E} \,\langle \nabla_{X} v , u \rangle_{E}
+ \|u\|_{E}^{2}\,\langle \nabla_{X} v , v \rangle_{E}
\right).
\end{aligned}
\end{equation}
By compatibility of linear connection $\langle \nabla_{X} v , v \rangle_{E} = \frac{1}{2}\,X\|v\|_{E}^{2}$, we can rewrite \eqref{c-d-general} as
\begin{equation*}
\begin{aligned}
c(X) &= \frac{1}{\det(M)}\left(
\|v\|_{E}^{2}\,\langle \nabla_{X} v , u \rangle_{E}
- \frac{1}{2}\,\langle u , v \rangle_{E} \,X\|v\|_{E}^{2}
\right),\\[0.2cm]
d(X) &= \frac{1}{\det(M)}\left(
-\,\langle u , v \rangle_{E} \,\langle \nabla_{X} v , u \rangle_{E}
+ \frac{1}{2}\,\|u\|_{E}^{2}\,X\|v\|_{E}^{2}
\right).
\end{aligned}
\end{equation*}

From relation \eqref{b+c=0} and using the relation obtained to $b(X), c(X)$, we have $b(X)+c(X)=0$ if and only if,
\begin{equation*}
\left(
-\frac{1}{2}\,\langle u , v \rangle_{E} \,X\|u\|_{E}^{2}
+ \|u\|_{E}^{2}\,\langle \nabla_{X} u , v \rangle_{E}
\right) +  \left(
\|v\|_{E}^{2}\,\langle \nabla_{X} v , u \rangle_{E}
- \frac{1}{2}\,\langle u , v \rangle_{E} \,X\|v\|_{E}^{2}
\right) =0
\end{equation*}
thus, 
\begin{equation}\label{step1ab}
\|u\|_{E}^{2}
\left(\langle\nabla_{X} u, v\rangle_{E} - \frac{1}{2} X\langle u,v\rangle_{E}\right)
+
\|v\|_{E}^{2}
\left(\langle\nabla_{X} v, u\rangle_{E} - \frac{1}{2} X\langle u,v\rangle_{E}\right)
=0    
\end{equation}
From compatibily of the linear connection, we have
$$
\frac12 X\langle u , v\rangle_{E} 
= \frac12\left[ \langle \nabla_{X} u , v\rangle_{E} + \langle u , \nabla_{X} v\rangle_{E} \right]
$$
Consequently,
\begin{eqnarray*}
    \langle\nabla_{X} u, v\rangle_{E} - \frac{1}{2} X\langle u,v\rangle_{E} &=& \langle \nabla_{X} u , v\rangle_{E} 
- \frac12 \langle \nabla_{X} u , v\rangle_{E} 
- \frac12 \langle u , \nabla_{X} v\rangle_{E}\\
&=& \frac12 \langle \nabla_{X} u , v\rangle_{E} - \frac12 \langle u , \nabla_{X} v\rangle_{E}.
\end{eqnarray*}
Similarly,
\begin{eqnarray*}
\langle\nabla_{X} v, u\rangle_{E} - \frac{1}{2} X\langle u,v\rangle_{E} &=& \langle \nabla_{X} v , u\rangle_{E} 
- \frac12 \langle \nabla_{X} u , v\rangle_{E} 
- \frac12 \langle u , \nabla_{X} v\rangle_{E} \\
&=& \frac12 \langle \nabla_{X} v , u\rangle_{E} - \frac12 \langle v , \nabla_{X} u\rangle_{E}.
\end{eqnarray*}
Thus, we can rewrite the equation \eqref{step1ab} as
\begin{align*}
\frac12 \|u\|_{E}^{2} \langle \nabla_{X} u , v\rangle_{E}
- \frac12 \|u\|_{E}^{2} \langle  u , \nabla_{X}v\rangle_{E}
+ \frac12 \|v\|_{E}^{2} \langle \nabla_{X} v, u\rangle_{E}
- \frac12  \|v\|_{E}^{2}\langle \nabla_{X}u , v\rangle_{E}
= 0
\end{align*}
rearranging the equation above,
\begin{align*}
\frac12 \|u\|_{E}^{2} \langle \nabla_{X} u , v\rangle_{E}
- \frac12  \|v\|_{E}^{2}\langle \nabla_{X}u , v\rangle_{E} 
+ \frac12 \|v\|_{E}^{2} \langle \nabla_{X} v, u\rangle_{E} - \frac12 \|u\|_{E}^{2} \langle \nabla_{X}v , u\rangle_{E}
&= 0\\
\frac12 \langle \nabla_{X} u , v\rangle_{E} \left( \|u\|_{E}^{2} 
- \|v\|_{E}^{2} \right) 
- \frac12 \langle \nabla_{X} v , u\rangle_{E} \left( \|u\|_{E}^{2} 
- \|v\|_{E}^{2} \right) 
&= 0\\
\frac{1}{2}\left( \langle \nabla_{X} u , v\rangle_{E} - \langle \nabla_{X} v , u\rangle_{E}\right) \left( \|u\|_{E}^{2} 
- \|v\|_{E}^{2} \right) & = 0
\end{align*}
In other words, for any vector field $X$ on $M$ the relations $c(X)+b(X)=0$ holds if and only if at least one of the following conditions is satisfied
$$\|u\|_{E}^{2} = \|v\|_{E}^{2} \quad \text{and} \quad \langle \nabla_{X} v , u\rangle_{E} = \langle \nabla_{X} u , v\rangle_{E}
$$
Suppose that $ \langle \nabla_{X} v , u\rangle_{E} = \langle \nabla_{X} u , v\rangle_{E}$. Then, from compatibility of linear connection, we have $\frac{1}{2}X\langle u,v \rangle_{E} =\langle \nabla_{X} u , v\rangle_{E}$. Then, we conclude
\begin{eqnarray*}
b(X) &=&
\frac{1}{\det(M)}\left(
-\frac{1}{2}\langle u , v \rangle_{E} X\|u\|_{E}^{2}
+ \|u\|_{E}^{2}\,\langle \nabla_{X} u , v \rangle_{E}
\right).\\
&=&
\frac{1}{\det(M)}\left(
-\frac{1}{2}\langle u , v \rangle_{E} X\|u\|_{E}^{2}
+ \frac{1}{2}\|u\|_{E}^{2}X\langle u,v \rangle_{E}
\right)\\
&=& 0
\end{eqnarray*}
From relation obtained in system \eqref{eq:U}, we have
$$
-\sum_i\int_{M}(\left|a(X_{i})\right|^{2}-\left|b(X_{i})\right|^{2}) = \lambda,
$$
but since $b(X)=0$ for any vector field $X$ on $(M,g)$, it follows that
$$
-\sum_i\int_{M}\left|a(X_{i})\right|^{2} = \lambda,
$$
but $|a(X_{i})|^{2}$ is positive for any vector field $X_{i}$ in local orthonormal frame of $M$, but this implies that 
$$
0 > - \sum_i\int_{M}\left|a(X_{i})\right|^{2} = \lambda \quad \text{and} \quad \lambda>0
$$
and that is only possible if 
$$
\sum_i\int_{M}\left|a(X_{i})\right|^{2}=0,
$$
and $\nabla_{X_{i}}u=0$ for any vector field $X_{i}$ in local orthonormal frame of $M$. Consequently, $\mathscr{Lu}= 0$, that is, $u$ is harmonic, but this is absurd since $u$ is non zero. Therefore, the relation 
$$
\langle \nabla_{X} v , u\rangle_{E} = \langle \nabla_{X} u , v\rangle_{E} \qquad \text{for all} \ u, v \in \Gamma(E) \ \text{and} \ X \in \Gamma(TM)
$$
cannot happen.

Let us now examine the second case where
\begin{equation}\label{rigidinsection}
\| u\|^{2}=\|v\|_{E}^{2}, \quad \text{for all} \ u,v \in \Gamma(E).
\end{equation}

To conclude the proof in this situation, we first establish the some auxiliary lemma. We will present the statement of the results and how they combine to obtain the conclusion, then we will present the proof of the results at the end of this section.

\begin{lemma}\label{wedge-infinitesimal}
Let $u,v$ be a smooth sections of vector bundle $E$ over a closed Riemannian manifold $(M,g)$ endowed with linear connection $\nabla$
compatible with the inner product $\langle \cdot, \cdot \rangle_{E}$ on the fibers. Assume that $u,v$ are eigensections of the Connection 
Laplacian operator 
$$
\mathscr{L}: \Gamma_{L^{2}}(E) \to \Gamma_{L^{2}}(E), \qquad\mathscr{L}=\nabla^{*}\nabla 
$$ 
satisfied the following properties: 
\begin{enumerate}[label=(\roman*)]
\item $u \wedge \nabla_{X}u = v \wedge \nabla_{X}v$, for any vector field $X$ defined on $(M,g)$.
\item $\|u(x)\|^{2}=\|v(x)\|^{2}$ for all $x \in M$.
\end{enumerate}
Then $u,v$ are infinitesimal rigid eigensections. More precisely, 
\begin{equation}\label{infinitesimalimportant}
\|\nabla u(x) \|_{x}^{2} = \|\nabla v(x) \|_{x}^{2}
\end{equation}
\end{lemma}

Thus, assuming that conditions \eqref{wedge-identity} and \eqref{rigidinsection} are satisfied, the Lemma \ref{wedge-infinitesimal} shows that $\| \nabla u\|_{x}^{2} = \| \nabla v\|_{x}^{2}$. 

In order to conclude the analysis, we next show that, since the relation \eqref{rigidinsection} holds, the analysis can be reduced to the case $\ell = 2$.

\begin{lemma}\label{lemmareduction} Let $M$ be a compact Riemannian manifold without boundary. Consider $E$ be a vector bundle over $M$ of rank $\ell \geq 2$ endowed with linear connection $\nabla$ compatible with the inner product on the fibers. Let $u,v$ be a smooth sections of vector bundle $E$. If $u(x)$ e $v(x)$ are linearly independent for al $x\in M$ and for any vector field $X \in \Gamma(TM)$ holds 
\begin{equation}\label{wedgedifferentzero}
u\wedge \nabla_X u \;=\; v\wedge \nabla_X v \;\neq 0.
\end{equation}
Then, there exist a vector subbundle $\tilde{E}$ of rank $2$ such that the fibers of this bundle may be identified with $\mathrm{span}\{u(x),v(x)\}$ for all $x\in M$ and $u,v$ must be sections of a smooth vector bundle of rank equal to $2$.
\end{lemma}
The next lemma shows that the hypotheses \eqref{rigidinsection}, \eqref{wedge-identity} and \eqref{infinitesimalimportant} cannot hold simultaneously.
This incompatibility is a key step in the proof of the main lemma.

\begin{lemma}\label{simultaneously}
Let $M$ be a compact Riemannian manifold without boundary. Consider $E$ be a vector bundle over $M$ of rank $\ell \geq 2$ endowed with linear connection $\nabla$ compatible with the inner product on the fibers. Let $\mathscr{L}: \Gamma_{L^{2}}(E) \to \Gamma_{L^{2}}(E)$ be the Connection Laplacian operator acting on sections of $E$. Then $\mathscr{L}$ admits no nonzero eigensection which is simultaneously wedge rigid, rigid and infinitesimal rigid.
\end{lemma}

Assuming these three lemmas, the proof of the main lemma is complete. We begin the proof of the auxiliary results with the following proposition.
\begin{proposition}
Let $M$ be a compact Riemannian manifold without boundary. Consider $E$ be a vector bundle over $M$ of rank $\ell \geq 2$ endowed with linear connection $\nabla$ compatible with the inner product on the fibers. Let $u,v$ nonzero smooth sections satisfying the relation 
\begin{equation*}
u \wedge \nabla_{X} u = v \wedge \nabla_{X}v,
\end{equation*}
for any vector field $X \in \mathfrak{X}(M)$ defined on M. Then, $u(x),v(x)$ are smooth sections linearly independents.
\begin{proof}
Suppose, by contradiction, that $u(x), v(x)$ are linearly dependent. Then, there exists a scalar function $c \in C^{\infty}(M)$ such that 
$u(x)=c v(x)$. Thus, 
$$
\nabla_{X}u(x) = \nabla_{X}(cv(x)) = X(c)v(x) + c\nabla_{X}v(x)
$$
Taking the wedge product in both sides of the equation with $u$, we obtain 
\begin{eqnarray*}
\nabla_{X}u(x) \wedge u(x) &=& X(c)v(x)\wedge u(x) + c\,\nabla_{X}v(x) \wedge u(x)\\
&=& X(c)v(x)\wedge cv(x) + c\,\nabla_{X}v(x) \wedge cv(x)\\
&=& |c|^{2}(\nabla_{X}v(x) \wedge v(x))
\end{eqnarray*}
Since $\nabla_{X} v\wedge v= \nabla_{X} u\wedge u$, it follows that $\nabla_{X} v(x) \wedge v(x)= |c|^{2}(\nabla_{X}v(x) \wedge v(x))$ and $c= \pm1$. Consequently, $u(x)=\pm v(x)$, for all $x \in M$,and by the principle of unique continuation it must coincide everywhere, which is a contradiction since $u$ and $v$ are linearly independent. Therefore $u$ cannot be a scalar multiple of $v$, and $u$ and $v$ are linearly independent.
\end{proof}
\end{proposition}

With this proposition at hand, we can now prove the auxiliary lemmas stated above.

\begin{proof}[Proof of Lemma \ref{wedge-infinitesimal}]
From identity \eqref{wedge-identity}, for any smooth vector field $X$, we have
\begin{equation*}
u(x)\wedge\nabla_{X}u(x) = v(x)\wedge\nabla_{X} v(x), \qquad \text{for all} \ u, v \in \Gamma(E) \ \text{and} \ X \in \Gamma(TM)
\end{equation*}
Using the definition of the inner product of the wedge product, for a vector field $X$ on $M$, we obtain the relation
\begin{eqnarray*}
\|u(x)\wedge\nabla_{X}u(x)\|^{2}&=& \|v(x)\wedge\nabla_{X}v(x)\|^{2}\\
\|u(x)\|^{2}\|\nabla_{X}u(x)\|^{2} - \langle u(x),\nabla_{X}u(x)\rangle^{2} &=& \|v(x)\|^{2}\|\nabla_{X}v(x)\|^{2} - \langle v(x),\nabla_{X}v(x)\rangle^{2}
\end{eqnarray*}
Therefore,
\begin{eqnarray*}
\|u(x)\|^{2}\|\nabla_{X}u(x)\|^{2}
- \left\|\frac{1}{2}\nabla\| u(x)\|^{2}\right\|^{2}
&=&
\|v(x)\|^{2}\|\nabla_{X}v(x)\|^{2}
- \left\|\frac{1}{2}\nabla\| v(x)\|^{2}\right\|^{2}
\end{eqnarray*}
and using the relation $\|u\|_{E}^{2}=\|v\|_{E}^{2}$, we obtain the following relation
\begin{equation*}
\|u(x)\|^{2}\|\nabla_{X}u(x)\|^{2}
-
 \left\|\frac{1}{2}\nabla\| u(x)\|^{2}\right\|^{2}
=
\|u(x)\|^{2}\|\nabla_{X}v(x)\|^{2}
-
 \left\|\frac{1}{2}\nabla\| u(x)\|^{2}\right\|^{2}
\end{equation*}
Hence, 
\begin{equation*}
\|\nabla_{X}u(x)\|^{2}
=\|\nabla_{X}v(x)\|^{2}.
\end{equation*}
In other words, the rigid and wedge rigid hypothesis implies that the infinitesimal rigid eigenspace hypothesis.
\end{proof}

\begin{proof}[Proof of Lemma \ref{lemmareduction}]
Let $x\in M$ and $X\in \Gamma(TM)$. Since $u(x)\wedge \nabla_Xu(x)\neq 0$, it follows that $u(x)$ and $\nabla_Xu(x)$ are linearly independent and therefore define a $2$-dimensional  plane in the fiber $E_{x}$. More precisely
$$
\pi_{(x,y)} =\mathrm{span}\{u(x),\nabla_Xu(x)\}\subset E_{x}
$$
has dimension equal to $2$. Moreover, since
$u(x)\wedge \nabla_Xu(x)=v(x)\wedge \nabla_Xv(x)\neq 0$, it follows that
$$
\pi_{(x,y)}=\mathrm{span}\{u(x),\nabla_Xu(x)\}=\mathrm{span}\{v(x),\nabla_Xv(x)\}.
$$
Therefore, there exists a unique $2$-dimensional plane $\pi_{(x,y)}$ containing the vectors 
$$
u(x),\,v(x),\,\nabla_Xu(x),\,\nabla_Xv(x)
$$ 
On the one hand, $u(x)$ e $v(x)$ are linearly independent,  
$$
\tilde{E}_{x}=\mathrm{span}\{u(x),v(x)\}
$$
has dimension equal to $2$, where $\tilde{E}_{x}\subset E_{x}$ . On the other hand $u(x),v(x)\in \pi_{(x,y)}$ and $\dim{\pi_{(x,y)}}=2$. Therefore,
$$
\tilde{E}_x=\pi_{(x,y)}.
$$

Thus, we obtain a subbundle $\tilde{E}$ of $E$ whose fibers are given by the two-dimensional planes
$$
\tilde{E}_x=\pi_{(x,y)}=\mathrm{span}\{u(x),v(x)\} = \mathrm{span}\{u(x),\nabla_Xu(x)\}=\mathrm{span}\{v(x),\nabla_Xv(x)\}
$$
Now, given that $\tilde{E}$ is a subbundle of $E$, we can induced a linear connection $\tilde{\nabla}$ on $\tilde{E}$ by 
$$
\tilde{\nabla}_X u := \nabla_X u,
\qquad \forall u\in\Gamma(\tilde{E}),\, X \in \Gamma(TM),
$$
and as $u(x),v(x) \in \tilde{E}_{x}$, it follows that, $\nabla_{X}u(x), \nabla_{X}v(x)$ lies in the fiber $\tilde{E}_{x}$. Consequently, the subbundle $\tilde{E}$ is invariant by linear connection $\nabla$ and the induced linear connection $\tilde{\nabla}$ defined on $\tilde{E}$ coincides with $\nabla$. 

Since the Connection Laplacian operator is defined with respect to linear connection defined on the smooth vector bundle and the linear connections $\nabla, \tilde{\nabla}$ coincides, we have 
$$
\mathscr{L}^{\tilde{\nabla}} = \mathscr{L}^{\nabla} \quad \text{and} \quad
\mathrm{Spec}(\mathscr{L}^{\tilde{\nabla}}) = \mathrm{Spec}(\mathscr{L}^{\nabla})
$$
Thus, for every $\lambda\in\mathrm{Spec}(\mathscr{L}^{\tilde{\nabla}})$,
$$
\ker(\mathscr{L}^{\tilde{\nabla}}-\lambda I)=\ker(\mathscr{L}^{\nabla}-\lambda I),
$$
Hence, since the relation $\eqref{wedgedifferentzero}$ holds, the section $u,v$ must be sections of a rank-2 vector bundle. 
\end{proof}

\begin{proof}[Proof of Lemma \ref{simultaneously}]
Suppose that the Connection Laplacian operator has wedge rigid eigenspace, that is, for all $u,v \in \Gamma(E)$, it follows that
$$
u \wedge\nabla_{X}\,u = v \wedge\nabla_{X}\,v \quad \text{for any} \ X \in \Gamma(E)
$$
From Lemma \eqref{wedge-infinitesimal}, we conclude that $\mathscr{L}$ has rigid and infinitesimal rigid eigenspace.

Since $E$ is a vector bundle associated to $\mathrm{SO}(2)$-principal bundle $P$ over $M$, there exists a isometry between $\Gamma_{L^{2}}(E)$, the space of smooth sections of $E$ and $C^{\infty}(P, V)^{\mathrm{SO}(V)}$, the space of $\mathrm{SO}(V)$-equivariant functions. This allows us to translate the rigid and infinitesimal rigid to $\mathrm{SO}(V)$-equivariant eigenfunctions of Laplace-Beltrami operator $\Delta$ defined on total space $P$. In other words, given $\Phi, \Psi : P \to \mathbb{R}^{2}$ with $\|\Phi\|_{L^{2}}=\|\Psi\|_{L^{2}}=1$ satisfies the following conditions:
\vspace{0.2cm}
\begin{enumerate}[label=(\roman*)]
    \item $\|\Psi(x)\| = \|\Phi(x)\|$, for all $x \in P$,
    \item $\|d\Phi^{h}(X)
)\|_{\mathbb{R}^{m}}^{2} = \|d^{h}\Psi(X)\|_{\mathbb{R}^{m}}^{2}$,  for any vector field $X \in \Gamma(TM)$.
\end{enumerate}
Thus, we are now in a position to apply the Lemma \eqref{TeseMarrocos}, since all of its hypotheses are satisfied. Then, there exists an ortogonal transformation $T: \mathbb{R}^{2} \to \mathbb{R}^{2}$ such that $\Psi(x)=T\Phi(x)$, where $T$ is a constant matrix. But this implies that components of $\Phi$ and $\Psi$ are linearly dependents, which is absurd, given that $\Phi, \Psi$ are orthogonal $\mathrm{SO}(2)$-equivariant eigenfunctions of the Laplace-Beltrami operator $\Delta_{\tilde{g}}$ defined on $P$. Hence, $\mathscr{L}_{g}$ has no rigid or infinitesimal rigid eigenspace. 
\end{proof}
Thus, the second case \eqref{rigidinsection} is impossible, and we obtain a contradiction. Therefore, since the Connection Laplacian operator $\mathscr{L}$ depends analytically on $t$, the derivatives $\lambda_{1}'(0),\ldots,\lambda_{\ell}'(0)$ of the eigenvalue branches associated with a multiple eigenvalue $\lambda$ are precisely the eigenvalues of $\Pi\circ \mathscr{L}^{(1)}|_{E_{\lambda}}$, where $\Pi$ is a projection into eigenspace $E_{\lambda}$, the eigenspace associated to $\lambda$. Therefore, if $\Pi\circ \mathscr{L}^{(1)}|_{E_{\lambda}}$ is not a multiple of the identity, there exist $i\neq j$ such that $\dot{\lambda}_{i}(0)\neq \dot{\lambda}_{j}(0)$, which implies that $\lambda_{i}(t)\neq \lambda_{j}(t)$ for all $|t|<\varepsilon$, and this completes the proof of the Lemma \ref{lemmaconnection}.
\end{proof}
\end{lemma}

\section{Main Results}\label{sectionmainresults}

\begin{proof}[Proof of Theorem \ref{teoremaconexão}]

For a linear connection $\nabla \in \mathrm{Conn}(E)$ defined on a smooth vector bundle $E$ compatible with inner product $\langle \cdot, \cdot \rangle_{E}$ on the fibers, we define the following subsets of $\mathrm{Conn}(E)$ by
$$
\mathrm{Conn}(E)_{n} := \Big\{ \nabla \in \mathrm{Conn}(E) \;\Big|\; 
\text{the first $n$ eigenvalues of } \mathscr{L} \text{ on } \Gamma_{L^{2}}(E) \text{ are simple} \Big\}.
$$
and 
$$
\mathrm{Conn}(E)_\infty := \Big\{ \nabla \in \mathrm{Conn}(E) \ \Big| \ 
\text{all eigenvalues of } \mathscr{L} \text{ on } \Gamma_{L^{2}}(E) \text{ are simple} \Big\},
$$
which satisfies the relation
\begin{equation*}
\mathrm{Conn}(E)_\infty \subset \cdots \subset \mathrm{Conn}(E)_n \subset \mathrm{Conn}(E)_{n+1} \subset \cdots 
\subset \mathrm{Conn}(E)_1 \subset \mathrm{Conn}(E)_0 = \mathrm{Conn}(E),    
\end{equation*}
and
$$
\mathrm{Conn}(E)_{\infty} \;=\; \bigcap_{n=0}^{\infty} \mathrm{Conn}(E)_{n}.
$$

By the continuous dependence of simple eigenvalues under small perturbations of the connection, each set $\mathrm{Conn}(E)_{n}$ is open in $\mathrm{Conn}(E)$. 
Thus, to prove that $\mathrm{Conn}(E)_{\infty}$ is residual in $\mathrm{Conn}(E)$, 
it is sufficient to show that $\mathrm{Conn}(E)_{n+1}$ is dense in $\mathrm{Conn}(E)_{n}$ 
for all $n = 0,1,2,\dots$.

Fix $\nabla(0) \notin \mathrm{Conn}(E)_{n}$, and let $U$ be an open neighborhood of $\nabla(0)$. 
Since at least one of the eigenvalues of the Connection Laplacian $\mathscr{L}_{\nabla(0)}$ has multiplicity greater than one, we proceed to split it. By Lemma \ref{lemmaconnection}, there exists a perturbation of the form
$$
\gamma(t) = \nabla_{0} + t\mathscr{A}, \qquad \mathscr{A} \in \Omega^{1}(M,\mathrm{End}(E)),
$$
for which at least two of the eigenvalues associated to $\mathscr{L}(t)$ that originally coincided become distinct, as long as $t$ is sufficiently small. Moreover, those eigenvalues that were already simple remain simple under such perturbations. Also, for $t$ small enough, we can assume that none of the eigenvalues outside the original cluster will enter it after perturbation.

Let $t_{1}$ be chosen so that $\nabla(t_{1}) \in U$. If $\nabla(t_{1}) \in \mathrm{Conn}(E)_{n}$, we are done. If not, then in finitely many steps the repetition of this construction produces a perturbed connection
$$
\nabla_{N} = \nabla_{0} + t_{1}\mathscr{A}_{1} + \cdots + t_{N}\mathscr{A}_{N}
$$
belonging to $U \cap \Gamma_{n}$. Hence, $\Gamma_{n}$ is dense.  
\end{proof}

As consequence, we can obtain the following result: 

\begin{corollary}\label{corollaryfinal}
Let $E$ be a vector bundle over compact $M$ without boundary with rank $m\geq 2$ endowed with inner product on the fibers $\langle \cdot, \cdot \rangle_{E}$. Then, the set of the Riemannian metric such that the eigenvalues of operator $\mathscr{L}_{g_{0}, \nabla_{0}}$ is simples is a residual set in $\mathrm{Met}^{k}(g)$.

\begin{proof}
From Theorem \ref{teoremaconexão}, there exists a residual set of connection and a fixed Riemannian metric on base manifold $M$ such that all non zero eigenvalues of Connection Laplacian are simple.

Let $g_{0}$ be a Riemannian metric on $M$ such that the operator $\mathscr{L}_{g_{0}, \nabla_{0}}$ has simple spectrum. By the continuous dependence of the Connection Laplacian operator on the parameter $g_{0}$, there exist $\varepsilon > 0$ and an open neighborhood $\mathcal{W}$ of $g_{0}$ such that
$$
\mathrm{Spec}(\mathscr{L}_g) \cap (\lambda_0 - \varepsilon, \lambda_0 + \varepsilon) \neq \varnothing
\quad \text{for all } g \in \mathcal{W}.
$$
Let $g(t)$ be an analytic one-parameter family of Riemannian metrics connecting $g$ to $g_{0}$, given by
$$
g(t) = (1-t)g + t g_{0}, \qquad t \in [0,1].
$$
and consider the associated Connection Laplacian operator $\mathscr{L}_{g(t), \nabla_{0}}$.

Suppose that the eigenvalue $\lambda_{0}$ has multiplicity $\ell \ge 2$. By Kato’s Lemma, there exists $\delta > 0$ such that, for $|t| < \delta$, there exists at most $\ell$ analytic eigenvalue curves
$$
\lambda_{1}(t), \dots, \lambda_{\ell}(t),
$$
and $\ell$ corsponding eigensections curves. As $\lambda_{i}(t), \lambda_{j}(t)$ are analytic curves, then for any pair $i \neq j$, the difference $\lambda_{i}(t) - \lambda_{j}(t)$ is also an analytic function in $t$. 

If $\lambda_{i}(t)$ and $\lambda_{j}(t)$ are distinct analytic eigenvalue branches, then 
$\lambda_{i}(t) - \lambda_{j}(t)$ is not identically zero and can vanish only at isolated values of $t$.
Therefore, eigenvalue multiplicities can arise only at finitely many points in any compact parameter interval.

We now use the assumption that $g(1) = g_{0}$ has simple spectrum in the chosen spectral window. In particular, no eigenvalue multiplicity occurs at $t=1$. Consequently, no two eigenvalue curves can remain coincident up to $t=1$. This implies that there are only finitely many crossing times in $[t_{0},1]$ for any fixed $t_{0} > 0$. We may therefore choose a parameter $t^{*} \in (0,1)$, arbitrarily close to $1$, such that no eigenvalue crossing occurs in the interval
$$
|\lambda(t^{*}) - \lambda_{0}| < \varepsilon.
$$
For this choice of $t^{*}$, all eigenvalues of $\mathscr{L}_{g(t^{*})}$ in the interval are simple. In particular, this shows that there exist Riemannian metrics arbitrarily close to $g_{0}$ for which the spectrum is simple in the interval. The openness follows from the continuous dependence of eigenvalues on the Riemannian metric.
\end{proof}
\end{corollary}

\section{\texorpdfstring{$G$-simple spectrum}{G-simple spectrum}}

Let $(M,g)$ be a compact $n$-dimensional Riemannian manifold without boundary. 
Let $P \to M$ be a principal $G$-bundle, where $G$ is a compact Lie group, and let
$$
\sigma : G \to \mathrm{O}(V)
$$
be an irreducible orthogonal representation on a finite-dimensional real vector space $V$.

The representation $\sigma$ determines the associated vector bundle $E := P \times_{\sigma} V$
defined as the quotient of $P \times V$ by the equivalence relation
$$
(p, v) \sim (p \cdot g, \sigma(g^{-1})v),
\qquad \text{for all } p \in P,\; g \in G,\; v \in V.
$$
Since $\sigma$ is orthogonal representation, it induces a $G$-invariant inner product $\langle \cdot , \cdot \rangle_{E}$
on the fibers of the vector bundle $E$. Thus, there is a isometric correspondence between spaces $\Gamma_{L^{2}}(E)$ and 
$$
L^{2}(P,V_{\sigma})^{\sigma}
:=
\left\{
\Phi:P\to V_{\sigma}\; ;\;
\Phi(p\cdot g)=\sigma(g^{-1})\Phi(p)
\right\}.
$$

Let $\nabla^{\sigma}$ be a linear connection on $E$ compatible with the inner product $\langle \cdot, \cdot \rangle_{E}$ on the fibers. We set
$$
\mathscr{L}: \Gamma_{L^{2}}(E) \to \Gamma_{L^{2}}(E), \qquad \mathscr{L}:=(\nabla^{\sigma})^{*}(\nabla^{\sigma})
$$
the Connection Laplacian operator associated to linear connection $\nabla^{\sigma}$ given above. The linear connection $\nabla^{\sigma}$ defined on $E$ induces an Ehresmann connection $H^{\sigma}$ on total space $P$. Assume that the Lie group $G$ is endowed with bi-invariant metric. Thus, we can construct a Riemannian metric $\tilde{g}$ on total space $P$ such that the Connection Laplacian operator $\mathscr{L}$ is unitary equivalent to Laplace-Beltrami operator $\Delta_{\tilde{g}}$ acting on $L^{2}(P,V_{\sigma})^{\sigma}$.

From right action of compact Lie group $G$ on Riemannian manifold $(P, \tilde{g})$, we can induce a regular representation of Lie group $G$ on Hilbert space $L^{2}(P)$ given by
$$
(U_{g}\Phi)(p)=\Phi(p\cdot g), \qquad g\in G.
$$
Since $U_{g}$ is a action induced by an isometry, it follows that the Laplace-Beltrami operator commutes with the regular action. More precisely,
$$
\Delta_{\tilde{g}}(U_{g}\Phi)=U_{g}(\Delta_{\tilde{g}}\Phi), \qquad \forall g\in G.
$$
Since $G$ is a compact group acting by isometries on $P$, the induced action on 
$L^{2}(P)$ is unitary. Hence, by the Peter--Weyl Theorem, we obtain the 
orthogonal decomposition
$$
L^{2}(P)\cong \bigoplus_{\rho\in\widehat G}\, I(V_{\rho} ),
$$
where $I(V_{\rho} )=m_{\rho}V_{\rho}$, and $V_{\rho}$ ranges over the 
irreducible unitary representations of $G$ and $m_{\rho}\in\mathbb{N}\cup\{0\}\cup{\infty}$ 
denotes the multiplicity of $V_{\rho}$ in $L^{2}(P)$.

Moreover, since the Laplace-Beltrami operator $\Delta_{\tilde g}$ commutes with 
the action of $G$, each eigenspace is $G$-invariant. Thus, for every non-zero 
eigenvalue $\lambda$ of $\Delta_{\tilde g}$, the associated eigenspace
$$
E_{\lambda}:=\ker(\Delta_{\tilde g}-\lambda)\subset L^{2}(P)
$$
is a $G$-invariant subspace of $L^{2}(P)$. Consequently, it inherits a decomposition into 
irreducible representations:
$$
E_{\lambda} \cong \bigoplus_{\rho\in\widehat G} m_{\lambda,\rho}\,V_{\rho},
$$
where $m_{\lambda,\rho}$ denotes the multiplicity of the irreducible 
representation $V_{\rho}$ inside $E_{\lambda}$.
The induced Ehresmann connection on $P$ and the $G$-invariant Riemannian metric $\tilde{g}$ defined on $P$ imply that $\mathscr{L}$ corresponds to the restriction of $\Delta_{\tilde{g}}$ to the $G$-equivariant subspace with respect to $\sigma$, up to a constant $c\ge0$ (the Casimir eigenvalue associated to bi-invariant metric on compact Lie group $G$)
\begin{equation}\label{equationDeltaMathscrL}
\Delta_{\tilde{g}}|_{L^{2}(P,V_{\sigma})^{\sigma}} \cong \mathscr L + c\,\mathrm{Id}.    
\end{equation}
Hence, if $\mathscr L s=\mu s$ and isometric correspondence between $\Gamma_{L^{2}}(E)$ and 
$L^{2}(P,V_{\sigma})^{\sigma}$, then
$$
\Delta_{\tilde{g}}\Phi=(\mu+c)\Phi.
$$
Define
$$
\mathcal E_{\lambda,\sigma}
:=
\left\{
\Phi\in L^{2}(P,V_{\sigma})^{\sigma}
: \,\Delta_{\tilde{g}}\Phi=\lambda\Phi
\right\}.
$$
Then for $\lambda=\mu+c$,
\begin{equation}\label{dimEL}
\dim \mathcal E_{\lambda,\sigma}
=
\dim\ker(\mathscr L-\mu).
\end{equation}
and we have $m_{\lambda,\sigma} = \dim\mathcal{E}_{\lambda,\sigma}$.

\begin{lemma}\label{g-simples}
If $\mu$ is a simple eigenvalue of the Connection Laplacian operator $\mathscr{L}$ on $\Gamma(E)$, then $m_{\lambda,\sigma}=1.$ 
\begin{proof}
Since $\lambda$ is simple, we have $\dim\ker(\mathscr{L}-\mu)=1$. Hence,
$\dim\mathcal E_{\lambda,\sigma}=1$.
If $m_{\lambda,\sigma}\ge2$, then by previous lemma, it follows that
$\dim\mathcal E_{\lambda,\sigma}\ge2$, but this is a contradiction. Therefore $m_{\lambda,\sigma}=1$.
\end{proof}
\end{lemma}
An eigenvalue $\lambda$ of $\Delta_{\tilde{g}}$  is said to be $G$-simple if its associated eigenspace is irreducible as a real representation of $G$. The Riemannian metric for which the eigenvalues of the Laplace-Beltrami operator are $G$-simple is called $G$-simple Riemannian metric. Our objective is to analyze the generic occurrence of this property in the class of $G$-invariant submersion Riemannian metrics on principal bundles with totally geodesic fibers.
\begin{theorem}
Let $P$ be a principal $G$-bundle over a compact Riemannian manifold $(M,g)$, with compact Lie group $G$. Then, the set of all $C^{k}$-Riemannian metric such that all non-zero eigenvalues of Laplace-Beltrami operator $\Delta_{\tilde{g}}$ are $G$-simple is a residual.
\begin{proof}
Let $P$ be a principal $G$-bundle over a compact Riemannian manifold $M$, where $G$ is a compact Lie group. 
Let 
$$
\sigma : G \to \mathrm{O}(V)
$$
be an irreducible orthogonal representation of $G$ on a finite-dimensional vector space $V$. 

We denote by 
$$
E_{\sigma} = P \times_{\sigma} V
$$
the associated vector bundle determined by $\sigma$, constructed as in the previous section.

From Theorem \ref{teoremaconexão}, there exists a pair $(g_{0}, \nabla_{0})$ such that spectrum of Connection Laplacian operator is simple. Consequently, by Corollary \ref{corollaryfinal}, there exists a residual set of Riemannian metrics such that all non-zero eigenvalues for Connection Laplacian operator is simple. 
Since the relations \eqref{equationDeltaMathscrL} and \eqref{dimEL} holds, we conclude by Lemma \ref{g-simples} that there exists a $C^{k}$-Riemannian metric such that the eigenvalues of $\Delta_{\tilde{g}}|_{I(V_\rho)}$ are $G$-simple. By an argument analogous to that of Corollary 
\ref{corollaryfinal}, we ensure the existence of a residual set of $C^{k}$-Riemannian metrics such that all non-zero eigenvalues of the Laplace-Beltrami operator $\Delta_{\tilde{g}}|_{I(V_\rho)}$ are $G$-simple with respect to representation $\sigma$ fixed. 

This proves that the eigenvalue $\lambda$ is simple on each isotypic component, but not necessarily simple in the full eigenspace $E_{\lambda}$, since two non-equivalents distinct representation types may occur. We will show that the spectrum can be decomposed even in that situation. 

Consider the orthogonal irreducible represention $\sigma = \sigma_{1}\oplus\sigma_{2}$, where $\sigma_{1}$ and $\sigma_{1}$ are two irreducible representations not equivalents of group $G$ on the vector space $V_{\sigma} = V_{\sigma_{1}} \oplus V_{\sigma_{2}},$

Denote by
$$
E_\sigma := P \times_{\sigma} (V_1 \oplus V_2)
\;\cong\;
(P \times_{\sigma_1} V_1)\;\oplus\;(P \times_{\sigma_2} V_2)
\;=\;
E_{\sigma_1} \oplus E_{\sigma_2},
$$
the associated vector bundle over the compact Riemannian manifold $(M,g)$, where $\oplus$ on the right denotes the Whitney sum. We can endowed the associated vector bundle $E_{\sigma}$ with linear connection $\nabla^{\sigma}=\nabla^{\sigma_{1}}\oplus \nabla^{\sigma_{2}}$ compatible with inner product on the fibers, where 
$$
\nabla^{\sigma_{1}}: \Gamma_{L^{2}}(E_{\sigma_{1}}) \to \Omega^{1}(M, E_{\sigma_{1}}) \quad \text{and}\quad \nabla^{\sigma_{2}}: \Gamma_{L^{2}}(E_{\sigma_{2}}) \to \Omega^{1}(M, E_{\sigma_{2}})
$$

Fix now $\lambda \neq 0$, and let $
E_\lambda := \ker(\Delta_{\tilde g} - \lambda) \subset L^2(P)$ be the eigenspace associated to eigenvalue $\lambda$. Suppose that $E_{\lambda}$ contains $G$-invariant subspaces $W^{(i)}$ such that $W^{(i)}$ is $G$-isomorphic to $V_{\sigma_i}$ for $i=1,2$.
Equivalently, for each $i$ there exists a nonzero
$\sigma_i$-equivariant function $\Phi^{(i)} \in \mathcal{E}_{\lambda,\sigma_i}$,
$$
\Delta_{\tilde{g}}\Phi^{(i)} = \lambda \Phi^{(i)},
\qquad
\Phi^{(i)}(p\cdot g)=\sigma_i(g^{-1})\Phi^{(i)}(p).
$$

Let $\sigma=\sigma_1\oplus\sigma_2$ and
$V_\sigma = V_{\sigma_1} \oplus V_{\sigma_2}$.
Any $G$-equivariant functions $\Phi : P \to V_\sigma$ with respect to $\sigma$, decomposes uniquely as
$\Phi=(\Phi_1,\Phi_2)$ with $ \Phi_i : P \to V_{\sigma_i}$, where $V_{\sigma_2}$ is the component where $\Phi_{1}$ vanishes and
the $V_{\sigma_1}$ is the component where $\Phi_{2}$ vanishes.

Since $\sigma$ acts block-diagonally, the equivariance condition
$$
\Phi(p\cdot g)=\sigma(g^{-1})\Phi(p)
$$
is equivalent to
$$
\Phi_i(p\cdot g)=\sigma_i(g^{-1})\Phi_i(p), \qquad i=1,2.
$$
Consequently,
$$
\mathcal E_{\lambda,\sigma}
=
\mathcal E_{\lambda,\sigma_1}\oplus
\mathcal E_{\lambda,\sigma_2}.
$$

The $G$-equivariant function $\Phi$ with respect to $\sigma$, satisfies the relation $\Delta_{\tilde{g}}\Phi=\lambda\Phi$. Consequently, 
$$
\Delta_{\tilde{g}}\Phi = (\Delta_{\tilde{g}}\Phi_{1}, \Delta_{\tilde{g}}\Phi_{2})
$$
Hence, we can decompose the eigenspace associated to $\lambda$ with respect to irreducible representation $\sigma$ as
$$
\mathcal{E}_{\lambda,\sigma}\simeq \mathcal{E}_{\lambda,\sigma_1}\oplus
\mathcal{E}_{\lambda,\sigma_2},
$$
satisfying the relation
$$
\dim\mathcal{E}_{\lambda,\sigma}
=
\dim\mathcal{E}_{\lambda,\sigma_1}
+
\dim\mathcal{E}_{\lambda,\sigma_2}.
$$
Since $\Phi^{(1)} \in \mathcal{E}_{\lambda,\sigma_1}\setminus\{0\}$ and
$\Phi^{(2)} \in \mathcal E_{\lambda,\sigma_2}\setminus\{0\}$, it follows that
$$
\dim \mathcal{E}_{\lambda,\sigma_1} \ge 1 \quad \text{and} \quad \dim \mathcal{E}_{\lambda,\sigma_2} \ge 1.
$$ 
Hence, $\dim \mathcal{E}_{\lambda,\sigma} \ge 2$. 
Since $\ker(\mathscr{L}^\sigma-\lambda)\simeq
\mathcal{E}_{\lambda,\sigma}$ holds, then $\dim \ker(\mathcal{L}^\sigma-\lambda)\geq 2$.

Since $\mathcal{E}_{\lambda,\sigma}\simeq \mathcal{E}_{\lambda,\sigma_1}\oplus
\mathcal{E}_{\lambda,\sigma_2}$ holds, we decompose $\Phi$ according to the direct sum $V = V_1 \oplus V_2$. 
For each $p \in P$, write
$$
\Phi(p) = (\Phi_1(p), \Phi_2(p)),
$$
where $\Phi_i$ is obtained by composing $\Phi$ with the canonical projection onto $V_i$.

By Theorem \ref{teoremaconexão}, there exist a pair $(g_{0}, \nabla_{0})$ such that all nonzero eigenvalues of the Connection Laplacian is simple. Hence 
\begin{equation}\label{lastLE}
1=\dim\ker(\mathcal{L}^\sigma-\lambda)\simeq
\dim\mathcal{E}_{\lambda,\sigma}    
\end{equation}
and this implies that $\Phi = (\Phi_{1}, 0)$ or $\Phi = (0, \Phi_{2})$. But, it implies that $(\Phi_{1}, 0)$ and $(0, \Phi_{2})$ are linearly independents, then the dimension of $\dim\mathcal{E}_{\lambda,\sigma}$ is at least $2$, which implies a contradiction of equation \eqref{lastLE}.

We distinguish the case in which the fiber has dimension one. When $E$ is a smooth vector bundle with rank one, the Laplace-Beltrami operator $\Delta_{\tilde{g}}$ must be restricted to the space of $G$-invariant eigenfunctions defined on $P$ with values in $V$. 
The conclusion of this result follows directly from a result of Karen Uhlenbeck \cite{Uhlenbeck}.

The above arguments are classic arguments of splitting eigenvalues developed by Wilson and Bleecker in \cite{BleeckerWilson1980}, and the proof of density for the $C^{k}$-Riemannian metrics follows analogously to the case of linear connection present in Section \ref{sectionmainresults}.

\end{proof}
\end{theorem}
\section*{Appendix}\label{riemanniansubmersion}

\subsection*{Relation between the Connection Laplacian and Laplace-Beltrami operators}

In this section, we establish the connection between the Connection Laplacian operator acting on a vector bundle and the Laplace-Beltrami operator on a Riemannian manifold $M$. This was explored by Gérard Besson in \cite{Besson}, and by Bergery and Bourguignon in \cite{Bourguignon}.  We will assume that $M$ is a closed orientable Riemannian manifold. 

Let $\pi: E \to M$ be a smooth vector bundle of rank $m$ over $M$ of dimension $n$, endowed with a linear connection $\nabla$ that is compatible  with a metric $\langle \cdot, \cdot \rangle_{E}$ on the fibers $V$. Consider $\mathrm{SO}(V)$ the structure Lie group of this bundle. We construct a principal $\mathrm{SO}(V)$-bundle $P = \bigsqcup_{x \in M} \mathrm{SO}(\pi^{-1}(x))$ associated to $E$, where $\pi^{-1}(x)$ is the fiber over a point $x \in M$, and the projection map $\widetilde{\pi}: P \to M$ maps each element of $\mathrm{SO}(\pi^{-1}(x))$ to its base point $x$.

 More precisely, by local trivialization $\{(U_{\alpha}, \psi_{\alpha})\}$ of vector bundle $E$, we have that
$$
\psi_{\alpha} : E|_{U_{\alpha}} \to U_{\alpha} \times V
$$
is a bundle isomorphism. For each $x \in U_\alpha$, we define $\psi_{\alpha,x} := \psi_\alpha(x, \cdot)$, providing a frame of $\pi^{-1}(x)$. This determines the following isomorphism:
$$
\varphi_{\alpha,x} : \mathrm{SO}(V) \to \mathrm{SO}(\pi^{-1}(x)), \quad \varphi \mapsto \psi_{\alpha,x} \circ \varphi.
$$
Consequently, the trivialization of $P$ is given by 
$$
\varphi_{\alpha} : U_{\alpha} \times \mathrm{SO}(V) \to \widetilde{\pi}^{-1}(U_{\alpha}), \quad \varphi_{\alpha}(x, \varphi) = \psi_{\alpha,x} \circ \varphi.
$$
Thus, $E$ is identified as the quotient space $P\times_{\mathrm{SO}(V)}V$ and the following diagram holds
\begin{center}
\begin{tikzpicture}[node distance=3cm, auto, >=Latex,
    every node/.style={font=\normalsize},
    every path/.style={line width=0.4pt}
  ]
  \node (GammaE1) {$P \times V$};
  \node (GammaE2) [below of=GammaE1] {$P$};
  \node (Cpv) [right of=GammaE1] {$P \times_{\mathrm{SO}(V)} V$};
  \node (Cpvv) [right of=GammaE2] {$(M,g)$};
  
  \draw[->] (GammaE1) to node [swap] {$\mathrm{pr_{1}}$} (GammaE2);
  \draw[->] (GammaE1) to node {$q$} (Cpv);
  \draw[->] (Cpv) to node {$\pi$} (Cpvv);
  \draw[<-] (Cpvv) to node {$\widetilde{\pi}$} (GammaE2);
  \draw[->] (GammaE2) -- (Cpv)
    node[midway, sloped, above] {$q \circ \Phi$}
    node[midway, sloped, below] {$u \circ \tilde{\pi}$};
\end{tikzpicture}

\vspace{0.5em}
\text{Diagram \ref{diagra}}\label{diagra}
\end{center}
where $q$ is the quotient map.

From the diagram above, we can write the smooth vector bundle $E = P \times_{\mathrm{SO}(V)} V \to M$  as an associated bundle to principal $\mathrm{SO}(V)$-bundle $P$ over $M$. Moreover, there exists a one to one correspondence 
\begin{equation}\label{isometriaxi}
\xi: \Gamma(E) \xrightarrow{1:1} C^\infty(P, V)^{\mathrm{SO}(V)},
\end{equation}
where $C^\infty(P, V)^{\mathrm{SO}(V)}$ denotes the space of smooth $\mathrm{SO}(V)$-equivariant functions $\Phi: P \to V$ in the sense that $\Phi(y\cdot h) = h^{-1} \Phi(y)$, for all $y \in P, \ h \in \mathrm{SO}(V)$.

The map $\xi$ assigns to each section $u \in \Gamma(E)$ the unique equivariant function $\Phi \in C^\infty(P, V)^{\mathrm{SO}(V)}$ such that
$$
\Phi(y)=\xi(u)(y) = p^{-1}(u(\tilde{\pi}(y))),
$$
where $p$ is an isometric map $p: V\rightarrow E_{\tilde{\pi}(y)}$ given by $p(v)=q(y,v)$  and $p^{-1}$ is its inverse (see Example 5.2 in \cite{kobayashi} for further details).

Thus, we are able to translate the concepts of the linear connection $\nabla$ defined on $E$ to associated principal $\mathrm{SO}(V)$-bundle, using the
$C^{\infty}(M)$-module isomorphisms $\xi$. More specifically, for all $u \in \Gamma(E)$ and $X \in \mathfrak{X}(M)$, we have
\begin{eqnarray}\label{isometry}
\nabla_{X}(u) = \xi^{-1}(X^{h}(\xi(u)))
\end{eqnarray}
where $X^{h}$ is horizontal lift of the vector field $X$ from $M$ to $P$ (see Proposition 1.3 in \cite{kobayashi}).
The relation obtained in (\ref{isometry}) gives us a characterization of the Connection Laplacian operator in terms of a horizontal differential operator acting on $\mathrm{SO}(V)$-equivariant functions defined on $P$. More precisely, for $ \Phi \in C^{\infty}(P,V)^{\mathrm{SO}(V)}$, we define: 
\begin{equation}\label{horizontalnabla}
    \Delta_{h} \Phi = - \sum_i (X^{h}_{i} \circ X^{h}_{i} - D_{X^{h}_{i}} X^{h}_{i} ) \Phi.
\end{equation}

Since the relation (\ref{isometry}) holds, we obtain a relation between operators $\Delta_{h}$ defined above in (\ref{horizontalnabla}) and Connection Laplacian $\mathscr{L}_{g}$ given by 
\begin{eqnarray}\label{firsthorizontal}
\Delta_{h} \Phi &=& - \sum_{i} (X^{h}_{i}(\xi(\nabla_{X_{i}}(u))) - D_{X^{h}_{i}}(\xi(\nabla_{X_{i}}(u)))\\
&=& - \sum_{i} \xi(\nabla_{X_{i}}\nabla_{X_{i}}(u)) -\xi( \nabla_{\nabla_{X_{i}}}(u))\nonumber\\
&=& \xi(\mathscr{L}_{g}u).\nonumber
\end{eqnarray}

The operator $\Delta_h$ is called the \textit{Horizontal Laplacian}.  

In what follows, we shall always assume that our Lie group $\mathrm{SO}(V)$ is endowed with a bi-invariant Riemannian metric, and $P\rightarrow M$ equipped with an unique Riemannian metric $\tilde{g}$ such that it is a Riemannian submersion with totally geodesic fiber (see Theorem 3.5 in \cite{Vilms}). In \cite{Besson}, the authors showed that the Laplace-Beltrami operator $\Delta_{\tilde{g}}$ on $P$ can be decomposed into
$$\Delta_{\tilde{g}} = \Delta_v + \Delta_h$$
such that $[\Delta_v, \Delta_h] = 0$, and $\Delta_{v}$ is the \textit{Vertical Laplacian}, defined as
$$
(\Delta_{v} \Phi) = \left( \Delta_{\widetilde{\pi}^{-1}(x)} \left( \Phi \big\downarrow \widetilde{\pi}^{-1}(x) \right) \right)
$$
where the fiber $\widetilde{\pi}^{-1}(x)$ of $P$ through $x \in M$ is isomorphic to Lie group $\mathrm{SO}(V)$ and $\Delta_{h}$ is the \textit{Horizontal Laplacian} defined in (\ref{horizontalnabla}).

The following  statement is an immediate consequence of the previous construction.
\begin{lemma}
The Connection Laplacian agrees with the horizontal Laplacian $\Delta_{h}$. Moreover
\begin{equation}\label{connectionxhorizontal}
\xi^{-1}\mathscr{L}_{g} \xi \Phi= \Delta_{h} \Phi = (\Delta_{\tilde{g}} - \Delta_{v}) \Phi = \Delta_{\tilde{g}} \Phi - (\Delta_{\mathrm{SO}(V)})\Phi.
\end{equation}
\end{lemma}
\begin{remark}\label{remarkspectrum}
Since $\mathrm{SO}(V)$ is a Lie group equipped with a bi-invariant metric, chosen so that $\mathrm{SO}(V)$ is the space of direct orthonormal frames on the sphere $\mathbb{S}^{m-1}$. Then, we have
$$
(\Delta_{\mathrm{SO}(V)} \Phi)(p) = (m-1)\Phi(p).
$$
Consequently, the relation (\ref{connectionxhorizontal}) can be written as
\begin{equation*}
\mathscr{L}_{g}u \sim \Delta_h \Phi = \left(\Delta_{\widetilde{g}} - \Delta_{\mathrm{SO}(V)}\right)\Phi = \Delta_p \Phi - (m-1)\Phi.
\end{equation*}
In particular, given $u =\xi^{-1}(\Phi)$ an eigensection of the Connection Laplacian, the following relation holds
\begin{equation}\label{conjugation}
\lambda (\xi^{-1} \Phi) = \mathscr{L}_{g}(\xi^{-1} \Phi) = \xi^{-1}(\Delta_{h} \Phi) = \xi^{-1} (\Delta_p \Phi - (m-1)\Phi) =\left[\lambda_{\tilde{g}}-(m-1)\right](\xi^{-1} \Phi).
\end{equation}
Therefore, for all eigenvalue $\lambda_{\tilde{g}}$ associated to a $\mathrm{SO}(V)$-equivariant eigenfunction $\Phi$ of $\Delta_{\tilde{g}}$, we can conclude that $\lambda = \lambda_{\tilde{g}} - (m-1)$. 
\end{remark}
In this position, we can translate the analysis of the spectrum of the Connection Laplacian operator $\mathscr{L}_{g}$ into the spectrum of the Laplace-Beltrami operator $\Delta_{\tilde{g}}$ acting on $C^{\infty}(P,V)^{\mathrm{SO}(V)}$. Next, we detail the action of $\Delta_{\tilde{g}}$ on space $C^{\infty}(P,V)^{\mathrm{SO}(V)}$.

Consider $\Phi=(\varphi_{1}, \dots, \varphi_{m}) \in C^{\infty}(P,V)^{\mathrm{SO}(V)}$, where each map $\varphi_{i}: P \to \mathbb{R}$ is an $\mathrm{SO}(V)$-equivariant function defined on Riemannian manifold $P$. The action of the Laplace-Beltrami operator $\Delta_{g}$ on space $C^{\infty}(P,V)^{\mathrm{SO}(V)}$ is given by
\begin{equation}\label{laplaciancoordinaties}
\Delta_{\tilde{g}}\Phi = (\Delta_{\tilde{g}}\varphi_{1}, \Delta_{\tilde{g}}\varphi_{2}, \ldots, \Delta_{\tilde{g}}\varphi_{m})
\end{equation}
where $\Delta_{\tilde{g}}\varphi_{k} = \lambda\varphi_{k}$ for each $k \in \{1, \ldots, m\}$. In addition, the operator $\Delta_{\tilde{g}}$ is an essentially self-adjoint operator on the Hilbert space $L^{2}(P, V)^{\mathrm{SO}(V)}$ given by completion of $C^{\infty}(P, V)^{\mathrm{SO}(V)}$ with respect to the global $L^{2}$-inner product
\begin{equation}\label{normequivariant}
\|\Phi\|_{L^{2}}^{2} = \int_{P}\langle \Phi(x), \Phi(x) \rangle_{\mathbb{R}^{m}} \,  dv_{g},    
\end{equation}
where $\langle \cdot, \cdot\rangle_{\mathbb{R}^{m}}$ denotes the Euclidean inner product on vector space $V$. This structure allows us to obtain a spectral decomposition of this space: 
\begin{equation*}
L^{2}(P, V)^{\mathrm{SO}(V)} = \bigoplus_{\lambda} E_{\lambda}^{\mathrm{SO}(V)},   
\end{equation*}
where $ E_{\lambda}^{\mathrm{SO}(V)} = \{ \Phi \in L^{2}(P, V)^{\mathrm{SO}(V)} \mid \Delta_{\tilde{g}}\Phi = \lambda \Phi \}
$. As previously mentioned, there exists a one to one correspondence between the spaces $C^{\infty}(P, V)^{\mathrm{SO}(V)}$ and $\Gamma(E)$ given by $C^{\infty}(M)$-isomorphism $\xi$. In particular, this isomorphism extend to isometry between the spaces $L^{2}(P, V)^{\mathrm{SO}(V)}$ and and $\Gamma_{L^{2}}(E)$. Indeed, consider $\Phi \in L^{2}(P, V)^{\mathrm{SO}(V)}$ such that $\Phi = \xi(u)$, where $u \in \Gamma_{L^{2}}(E)$, from (\ref{normequivariant}), we have that 
\begin{eqnarray*}
\|\xi(u)\|_{L^{2}}^{2} &=& \int_{P}\langle \xi(u)(x), \xi(u)(x) \rangle_{\mathbb{R}^{m}} \,  dv_{g}\\
& = &  \int_{P}\langle p^{-1}(u(\tilde{\pi}(x)), p^{-1}(u(\tilde{\pi}(x)) \rangle_{\mathbb{R}^{m}} \,  dv_{g}.
\end{eqnarray*}
Since $p^{-1}$ is an isometry, we have that 
$$
\int_{P}\langle p^{-1}(u(\tilde{\pi}(x)), p^{-1}(u(\tilde{\pi}(x)) \rangle_{\mathbb{R}^{m}} \,  dv_{g} = \int_{P}\langle u(\tilde{\pi}(x), u(\tilde{\pi}(x) \rangle_{E} \,  dv_{g}.
$$
Thus, given that $P$ is a principal $\mathrm{SO}(V)$-bundle, it follows that every point $y \in M$ is the image of all points $x \in P$ such that $\tilde{\pi}(x)=y$. Thus, $\|\xi(u)\|_{L^{2}}^{2} = \|u\|_{L^{2}(E)}^{2}$.

\begin{remark} The relation between the spaces $L^{2}(P, V)^{\mathrm{SO}(V)}$ and $\Gamma_{L^{2}}(E)$ by isometry above and spectral decomposition properties allows us to conclude that an eigenvalue $\lambda$ of the Connection Laplacian operator $\mathscr{L}_{g}$ is simple if and only if the eigenvalue $\lambda_{\tilde{g}}$ of the Laplace-Beltrami operator $\Delta_{\tilde{g}}$ acting on $L^{2}(P, V)^{\mathrm{SO}(V)}$ is $\mathrm{SO}(V)$-simple.  
    
\end{remark}

\begin{lemma}\label{TeseMarrocos}Under the same hypotheses as above, assume in addition that $\Phi, \Psi : P \to \mathbb{R}^{2}$ with $\|\Phi\|_{L^{2}}=\|\Psi\|_{L^{2}}=1$ satisfies the following conditions:
\vspace{0.2cm}
\begin{enumerate}[label=(\roman*)]
    \item $\|\Psi(x)\| = \|\Phi(x)\|$ for all $x \in P$,\\[0.01cm]
    \item $\|d\Phi^{h}(X)
)\|_{\mathbb{R}^{m}}^{2} = \|d^{h}\Psi(X)\|_{\mathbb{R}^{m}}^{2}$ for any vector field $X \in \Gamma(TM)$.
\end{enumerate}
\vspace{0.2cm}
Then, there exists an isometry $A: \mathbb{R}^{2} \to \mathbb{R}^{2}$ such that
$$
\Phi(x) = A \Psi(x), \quad \text{for all} \ x \in M.
$$

\begin{proof}
Let $\Psi, \Phi$ be two $\mathrm{SO}(2)$-equivariant functions defined on Riemannian submersion with totally geodesic fibers $P$. From rigid hypothesis, we have that $\|\Phi(x)\|_{\mathbb{R}^{2}}^{2}=\|\Psi(x)\|_{\mathbb{R}^{2}}^{2}$ for all $x\in P$. Consequently, there exists an element $T(x) \in \mathrm{SO}(2)$ given by 
\begin{equation*}
T(x) = 
\begin{bmatrix}
\cos \theta(x) & -\sin \theta(x) \\
\sin \theta(x) & \cos \theta(x)
\end{bmatrix}.
\end{equation*}
such that $\Psi(x)=T(x)\Phi(x)$, for $x \in P$. From $\mathrm{SO}(2)$-equivariance hypothesis, we have
$$
\Psi(x\cdot h) = T(x\cdot h)\Phi(x\cdot h) = T(x\cdot h)h^{-1}\Phi(x)
$$
On the other hand,
$$
\Psi(x\cdot h) = h^{-1}\Psi(x) = h^{-1}T(x)\Phi(x)
$$
Therefore, surely
$$
T(x\cdot h) h^{-1}\Phi(x) = h^{-1} T(x) \Phi(x).
$$
Since $\Phi(x)\neq 0$ for a dense set of $P$, it follows that $T(x\cdot h) = h^{-1}T(x)h$, for all $h \in \mathrm{SO}(2)$. Moreover, the Lie group $\mathrm{SO}(2)$ is abelian, and we conclude that $T$ is $\mathrm{SO}(2)$-invariant.

We will prove the infinitesimal rigid hypothesis in vertical direction, that is, given $\mathrm{SO}(2)$-equivariant functions $\Psi, \Phi$ with $\|\Psi \|=\|\Phi\|=1$, we have
$$
\sum^{m}_{i=1}\|d\Psi(Y_{i}) \|^2= \sum^{m}_{i=1}\|d\Phi(Y_{i}) \|^2
$$
where $Y_{i}$ belongs to the local vertical orthonormal frame on $P$. 
For this, we must prove that rigid eigenspace hypothesis implies infinitesimal rigid of the eigenspaces in vertical direction.  

Let $\Phi, \Psi$ be a $\mathrm{SO}(2)$-equivariant function defined on $P$ with values in $V $ and consider 
$$
\gamma : (0,1) \to \mathrm{SO}(2) \ \text{a smooth curve such that} \ \gamma(0)=e, \ \dot{\gamma}(0) = X_{e} \in \mathfrak{so}(2).
$$ 

We can move any point $x \in P$ along to $\gamma$ by relation $x(t) = x \cdot \gamma(t)$. 

Since $\Phi(x)=T(x)\Psi(x)$, we have 
$$
\langle \Phi(x(t)), \Phi(x(t)) \rangle = \langle T(x(t))\Psi(x(t)), T(x(t))\Psi(x(t)) \rangle
$$
and differentiating with respect to $t=0$
$$
\|\dot{\Phi}(x) \|^{2}=\langle \dot{\Phi}(x), \dot{\Phi}(x) \rangle = \langle T(x)\dot{\Psi}(x), T(x)\dot{\Psi}(x) \rangle = \|\dot{\Psi}(x)\|^{2} 
$$
Consequently, we obtain the infinitesiam rigid hypothesis in any directions $\tau = \frac{\partial}{\partial x_{i}} + \frac{\partial}{\partial x_{j}}$ belongs to $T_{x}P$, that is 
$$
\left\|\frac{\partial}{\partial \tau}\Psi(x)  \right\| = \left\|\frac{\partial}{\partial \tau}\Phi(x)  \right\|
$$

In what follows, we shall analyze the derivative of the $\mathrm{SO}(2)$-equivariant eigenfunction $\Psi$ with respect to $\tau$.

Differentiating $\Psi(x)=T(x)\Phi(x)$ with respect to $\tau$, we have 
\begin{equation*}
\frac{\partial}{\partial \tau}\Psi(x) = \frac{\partial}{\partial \tau}\left[T(x)\Phi(x)\right] = \left[\frac{\partial}{\partial \tau}T(x)\right]\Phi(x) + T(x)\left[\frac{\partial}{\partial \tau}\Phi(x)\right]
\end{equation*}

Consequently, 
\begin{eqnarray*}
\left\|\frac{\partial}{\partial \tau}\Psi(x)\right\|^{2} 
&=&  \left\|\left(\frac{\partial}{\partial \tau}T(x)\right)\Phi(x)\right\|^{2} + 2\left\langle T(x)^{t}\left(\frac{\partial}{\partial \tau}T(x)\right)\Phi(x), \left(\frac{\partial}{\partial \tau}\Phi(x)\right)  \right\rangle + \left\|\left(\frac{\partial}{\partial \tau}\Phi(x)\right)\right\|^{2}.
\end{eqnarray*}

On other hand, by denoting
$$
J = \begin{bmatrix}
0 & -1 \\
1 & 0
\end{bmatrix}
$$
we obtain
$$
\frac{\partial T(x)}{\partial \tau}  = \frac{\partial \theta(x)}{\partial \tau} J T(x).
$$
Consequently, we obtain
\begin{eqnarray*}
\left\|\frac{\partial}{\partial \tau}\Psi(x)\right\|^{2} &=&  \left\|\left(\frac{\partial}{\partial \tau}T(x)\right)\Phi(x)\right\|^{2} + 2\left\langle T(x)^{t}\left(\frac{\partial \theta(x)}{\partial \tau}J T(x)\right)\Phi(x), \left(\frac{\partial}{\partial \tau}\Phi(x)\right)  \right\rangle \left\|\left(\frac{\partial}{\partial \tau}\Phi(x)\right)\right\|^{2}\\
&=&  \left\|\left(\frac{\partial \theta(x)}{\partial \tau} J T(x)\right)\Phi(x)\right\|^{2} + 2\left\langle\left(\frac{\partial \theta(x)}{\partial \tau}J\right)\Phi(x), \left(\frac{\partial}{\partial \tau}\Phi(x)\right)  \right\rangle + \left\|\left(\frac{\partial}{\partial \tau}\Phi(x)\right)\right\|^{2}
\end{eqnarray*}
From infinitesimal rigid hypothesis, we can rewrite the equation above as 
\begin{eqnarray*}
\frac{\partial \theta(x)}{\partial \tau} \cdot \left(\frac{\partial \theta(x)}{\partial \tau}\left\|\Phi(x)\right\|^{2} + 2\left\langle J\Phi(x), \left(\frac{\partial}{\partial \tau}\Phi(x)\right)  \right\rangle\right) &=&0.
\end{eqnarray*}
Now, we will analyze the identity above. If $\frac{\partial \theta(x)}{\partial \tau}=0$ it follows that $\theta$ is constant, we are done. Suppose then that 
$$
\frac{\partial \theta(x)}{\partial \tau}\left\|\Phi(x)\right\|^{2} + 2\left\langle J\Phi(x), \left(\frac{\partial}{\partial \tau}\Phi(x)\right)\right\rangle =0
$$
we have that 
$$
\frac{\partial \theta(x)}{\partial \tau} = \frac{-2\left\langle J\Phi(x), \left(\frac{\partial}{\partial \tau}\Phi(x)\right)\right\rangle}{\left\|\Phi(x)\right\|^{2}}
$$
Consider $\Phi=(\varphi_{1}, \varphi_{2})$, where each $\varphi_{i}: P \to \mathbb{R}$ is an $\mathrm{SO}(2)$-equivariant function with $i=1,2$. Then,
\begin{eqnarray*}
    \left\langle J\Phi(x), \left(\frac{\partial}{\partial \tau}\Phi(x)\right)\right\rangle &=& \left\langle (-\varphi_{2}(x), \varphi_{1}(x)), \left(\frac{\partial \varphi_{1}(x)}{\partial \tau}, \frac{\partial \varphi_{2}(x)}{\partial \tau}\right)\right\rangle\\
    & = & -\varphi_{2}(x)\frac{\partial \varphi_{1}(x)}{\partial \tau}+\varphi_{1}(x)\frac{\partial \varphi_{2}(x)}{\partial \tau}
\end{eqnarray*}
Thus the above equation can be rewritten as
\begin{equation}\label{differentialtheta}
\frac{\partial \theta(x)}{\partial \tau} = \frac{-2\left(\varphi_{2}(x)\frac{\partial \varphi_{1}(x)}{\partial \tau}-\varphi_{1}(x)\frac{\partial \varphi_{2}(x)}{\partial \tau}\right)}{\varphi_{1}(x)^{2} + \varphi_{2}(x)^{2}} = -2\frac{\partial}{\partial \tau}\arctan{\left(\frac{\varphi_{1}(x)}{\varphi_{2}(x)}\right)}
\end{equation}
From relation (\ref{differentialtheta}), we have
\begin{equation*}
    \frac{\partial \theta(x)}{\partial \tau} + 2\frac{\partial}{\partial \tau}\arctan{\left(\frac{\varphi_{1}(x)}{\varphi_{2}(x)}\right)} = 0
\end{equation*}
Hence, 
\begin{equation*}
\theta(x) = -2\arctan{\left(\frac{\varphi_{1}(x)}{\varphi_{2}(x)}\right)} + C.
\end{equation*}
where $C$ is a constant.

Now, we will determine the matrix $T(x)$ explicitly. To simplify the calculations, let us assume the following notation
$$
\theta(x) = g(x) + C, \quad \text{where} \quad g(x) := -2 \arctan\left( \frac{\varphi_{1}(x)} {\varphi_{2}(x)} \right).
$$

Consequently, the matrix $T(x)$ can be written as
$$
T(x) = 
\begin{bmatrix}
\cos{C} & -\sin{C} \\
\sin{C} & \cos{C}
\end{bmatrix}
\begin{bmatrix}
\cos{g(x)} & -\sin{g(x)}\\
\sin{g(x)} & \cos{g(x)}
\end{bmatrix}.
$$

that is a total rotation by $\theta(x) = g(x) + C$ in $\mathbb{R}^{2}$, which can be interpreted in complex plane as
$$
\exp{i\theta(x)} = \cos{\theta(x)} + i\sin{\theta(x)}.
$$
Applying this to the complex vector $\Phi(x) = \varphi_{1}(x) + i\varphi_{2}(x)$, we obtain the relation
$$
\Psi(x) := \exp\left\{{i\theta(x)}\right\} \Phi(x), \quad \text{where} \  \theta(x) = -2 \arctan\left( \frac{\varphi_{1}(x)} {\varphi_{2}(x)}\right) + C.
$$
Therefore, we can rewrite:
$$
\Psi(x) = \exp{\left\{i \arctan\left(\frac{\varphi_{1}(x)} {\varphi_{2}(x)}\right) + iC\right\}} \ \Phi(x) = \exp\{{iC}\} \exp{\left\{-2i \arctan\left(\frac{\varphi_{1}(x)} {\varphi_{2}(x)}\right)\right\}} \Phi(x).
$$
Given that
$$
\exp{\left\{-2i \arctan\left(\frac{\varphi_{1}(x)} {\varphi_{2}(x)}\right)\right\}} = \left( \exp{\left\{i \arctan\left(\frac{\varphi_{1}(x)} {\varphi_{2}(x)}\right)\right\}} \right)^{-2},
$$
it follows that
\begin{equation}\label{psiiC}
\Psi(x) = \exp\left\{{iC}\right\} \cdot \frac{\Phi(x)} {\left( \exp{\left\{i \arctan\left(\frac{\varphi_{1}(x)} {\varphi_{2}(x)}\right)\right\}} \right)^{2}}.    
\end{equation}
Thus, using the polar form of $\Phi(x)$, we obtain
$$
\Phi(x) = \|\Phi(x)\| \exp{\{i\phi(x)\}} \quad \Longrightarrow \quad \frac{\Phi(x)} {\|\Phi(x)\|} = \exp{\{i\phi(x)\}}.
$$
Therefore,
$$
\exp{\left\{i \arctan\left(\frac{\varphi_{1}(x)} {\varphi_{2}(x)}\right)\right\}} = \left( \frac{\Phi(x)} {\|\Phi(x)\|} \right) \quad \Longrightarrow \quad \exp{\left\{-2i \arctan\left(\frac{\varphi_{1}(x)} {\varphi_{2}(x)}\right)\right\}}= \left( \frac{\overline{\Phi(x)}}{\|\Phi(x)\|} \right)^{2}
$$
Since $\exp\{{-2i\phi}\} = \left( \overline{\exp{\{i\phi\}}} \right)^{2}$, we can rewrite the relation (\ref{psiiC}) as 
$$
\Psi(x) = \exp{\{iC\}} \cdot \left(\frac{\overline{\Phi(x)}}{\|\Phi(x)\|} \right)^{2} \cdot \Phi(x) = \exp{\{iC\}} \cdot \frac{\overline{\Phi(x)}^{2}}{\|\Phi(x)\|^{2}} \cdot \Phi(x).
$$
In conclusion, given that
\begin{equation*}
\frac{\overline{\Phi(x)}^{2}}{\|\Phi(x)\|^{2}} \cdot \Phi(x) = \overline{\Phi(x)}
\end{equation*}
holds, we have $\Psi(x) = \exp{\{iC\}} \cdot \overline{\Phi(x)}$. Therefore, the orthogonal transformation $T$ is given by
$$
T(x)= 
\begin{bmatrix}
\cos{C}  & -\sin{C} \\
\sin{C} & \cos{C}
\end{bmatrix}\begin{bmatrix}
0  & -1 \\
1 & 0
\end{bmatrix}.
$$ 
\end{proof}
\end{lemma}

\backmatter

\bmhead{Acknowledgements}
This research was partially supported by Coordenação de Aperfeiçoamento de Pessoal de Nível Superior (CAPES) and partially supported by Fundação de Amparo à Pesquisa do Estado de São Paulo-FAPESP 2020/14075-6.

\bibliography{laplacian_connection_Principal_Bundle}

\end{document}